\documentclass[a4paper]{amsart} 
\usepackage[utf8]{inputenc}
\usepackage[ngerman, english]{babel}
\usepackage{color,soul}
\usepackage{enumerate}
\usepackage{hyperref} 
\usepackage{cleveref}
\usepackage{tabularx}
\usepackage{float}

\newtheorem{theorem}{Theorem}[section]
\newtheorem{definition}[theorem]{Definition}
\newtheorem{corollary}[theorem]{Corollary}
\newtheorem{proposition}[theorem]{Proposition}
\newtheorem{lemma}[theorem]{Lemma}
\newtheorem{remark}[theorem]{Remark}
\newtheorem{example}[theorem]{Example}

\newcommand{\R}{\mathbb{R}}
\newcommand{\C}{\mathbb{C}}
\newcommand{\N}{\mathbb{N}}

\newcommand{\K}{\mathbb{K}}

\newcommand{\ot}{\otimes}

\newcommand{\bfK}{{\mathbf K}}

\DeclareMathOperator{\cone}{cone}
\DeclareMathOperator{\conv}{conv}

\DeclareMathOperator{\tr}{Tr}
\DeclareMathOperator{\maxsym}{MaxSym}

\newcommand{\oR}{{\mathbb R}}

\newcommand{\oN}{{\mathbb N}}

\newcommand{\oS}{\mathbb {S}}

\newcommand{\ui}{\underline i}
\newcommand{\uj}{\underline j}
\newcommand{\uk}{\underline k}

\newcommand{\MQ}{{\mathcal Q}}
\newcommand{\MP}{{\mathcal P}}

\newcommand{\bfi}{{\mathbf i}}
\newcommand{\MFS}{\mathfrak S}
\newcommand{\vvec}{\text{vec}}
\newcommand{\PiMS}{\Pi_{\text{\rm MaxSym}}}
\newcommand{\MaxSym}{\text{\rm MaxSym}}

\newcommand{\ub}{\text{\rm ub}}

\newcommand{\Tr}{\text{\rm Tr}}

\newcommand{\sos}{\text{\rm sos}}
\newcommand{\spec}{\text{\rm sp}}
\newcommand{\ignore}[1]{}

\newcommand{\johnstonlovitz}{Lovitz and Johnston}
\newcommand{\val}{\text{\rm val}}

\renewcommand{\epsilon}{\varepsilon} 
\newcommand{\MoL}{}
\newcommand{\ATB}{}
\newcommand{\tcolred}{}
\newcommand{\tcolblue}{}

\usepackage{pifont}
\subjclass{90C23, 
	90C22, 
	81P45 
}
\keywords{polynomial optimization, sums of squares, quantum de Finetti, semidefinite programming, spectral bounds, moment problem} 

\begin{document}

\begin{abstract} 
We revisit the convergence analysis of two approximation hierarchies  for polynomial optimization on the unit sphere. The first one is based on the moment-sos approach and gives semidefinite bounds for which Fang and Fawzi (2021) showed an analysis in $O(1/r^2)$ for the $r$th level bound, using the polynomial kernel method. The second hierarchy was recently proposed by \johnstonlovitz{} (2023) and gives spectral bounds for which they show a convergence rate in $O(1/r)$, using a quantum de Finetti theorem of Christandl et al. (2007) that applies to complex Hermitian matrices with a ``double'' symmetry. We investigate links between these approaches, in particular, via duality of moments and sums of squares. 

Our main results include showing  that the spectral bounds cannot have a convergence rate better than $\mathcal{O}(1/r^2)$ and that they do not enjoy generic finite convergence. In addition, we propose alternative performance analyses  that involve explicit constants depending  on intrinsic parameters of the optimization problem. For this we develop a novel ``banded'' real de Finetti theorem that applies to real matrices with ``double" symmetry. We also show how to use the polynomial kernel method to obtain a de Finetti type result \tcolblue{in $O(1/r^2)$} for real maximally symmetric matrices, improving an earlier result \tcolblue{in $O(1/r)$} of Doherty and Wehner (2012). 
\end{abstract}

\title[Convergence analysis via de Finetti theorems]{Moment-sos and spectral hierarchies for polynomial optimization on the sphere and quantum de Finetti theorems}
\date{\today}
\author[Blomenhofer]{Alexander Taveira Blomenhofer}
\author[Laurent]{Monique Laurent}
\thanks{University of Copenhagen}
\thanks{Centrum Wiskunde \& Informatica (CWI), and Tilburg University}
\maketitle

\section{Introduction}\label{sec:introduction}

Throughout we consider the polynomial optimization problem, asking to minimize a polynomial $p$ over the unit sphere $\oS^{n-1}$, i.e., to compute
\begin{align}\label{eqpmin}
p_{\min}=\min_{x\in \oS^{n-1}}p(x)=\min\{p(x): x\in \R^n, \|x\|_2=1\}.
\end{align}
Therein, $p\in \R[x]=\R[x_1,\ldots,x_n]$ is an $n$-variate polynomial of degree $2d$. In other words,
$p_{\min}$ is the largest scalar $\lambda$ for which the polynomial $p -\lambda $ is nonnnegative on the unit sphere $\oS^{n-1}$. Replacing this nonnegativity condition by existence of a sum-of-squares representation leads to the well-known sum-of-squares (sos, for short) bounds: for an integer $r\ge d$,
\begin{align}\label{eqsosr}
\begin{split}
\sos_r(p)=\max_{\lambda\in \R} \{\lambda : p-\lambda =\sigma +u(1-\|x\|_2^2),
\  \sigma\in \R[x]_{\le 2r} \text{ sos},\ \tcolblue{u\in \R[x]_{2r-2}}\}.
\end{split}
\end{align}
The  parameters $\sos_r(p)$ provide a hierarchy of lower bounds that converge asymptotically to $p_{\min}$ (Lasserre \cite{Lasserre2001}). They are also often referred to as the moment-sos bounds (see Lemma \ref{lemsosmoment} for their `moment' reformulation).

\medskip
An alternative hierarchy of lower bounds $\spec_r(p)$ on $p_{\min}$ has been recently introduced by \johnstonlovitz{} 
 \cite{Johnston_Lovitz_Vijayaraghavan_2023} for homogeneous polynomials $p$, that are based on a (generalized) eigenvalue computation. 
Roughly speaking, these spectral bounds rely on very specific sums of squares representations of $ p-\lambda $ on the sphere, based on \emph{doubly symmetric} representing matrices. We refer to Section \ref{secJLV} for the precise definition of the spectral bounds. 
We have $\spec_r(p)\le \sos_r(p)\le p_{\min}$. So, the  sos bounds are stronger, but they are also more costly to compute than $\spec_r(p)$; indeed the latter boil down to an eigenvalue computation while the former rely on semidefinite programming. 

\medskip
Convergence analysis has been carried out for both hierarchies. Interestingly, there are common ingredients 
to both cases, like the the use of the polynomial kernel method  and the relevance to Quantum de Finetti type results.  We discuss the polynomial kernel method in \Cref{secrealQdF} and introduce the relevant Quantum de Finetti theorems in \Cref{secQdF}. However, there are also some important differences, as we now briefly sketch. 

\subsection*{Convergence analysis}

\johnstonlovitz{} \cite{Johnston_Lovitz_Vijayaraghavan_2023} show the following convergence rate $p_{\min}-\spec_r(p)=O(1/r)$ for their spectral bounds (see Theorem \ref{theoJLV}).
A crucial ingredient in their proof is the Quantum de Finetti (abbreviated as QdF) theorem of Christandl et al. \cite{Christandl_2007_deFinetti_oneandhalf} for {\em doubly symmetric} complex Hermitian matrices (see Theorem~\ref{theoQdF}). Since this result is stated for complex matrices, additional work is needed to be able to reach a conclusion for the spectral bounds, which deal with \emph{real} polynomials. The strategy adopted in \cite{Johnston_Lovitz_Vijayaraghavan_2023} is to establish a relation between the real optimization problem (\ref{eqpmin})  and its complex analog (asking to minimize a polynomial on the complex unit sphere) to which one can then apply the above mentioned QdF theorem. The $O(1/r)$ rate in the QdF theorem results in the $O(1/r)$ convergence rate for the spectral bounds $\spec_r(p)$.

\medskip
The convergence rate of the sos bounds in (\ref{eqsosr}) has been analyzed by Fang and Fawzi \cite{Fang_Fawzi_2020} who show  $p_{\min}-\sos_r(p)=O(1/r^2)$ when $p$ is   homogeneous (see Theorem~\ref{theoFF}). This improves the earlier convergence analyses in $O(1/r)$ by Reznick \cite{Reznick_1995} and Doherty and Wehner~\cite{Doherty_Wehner_2013}. 

These results follow roughly a similar approach, but there are some subtle differences. A common ingredient in both proofs is the use of the {\em polynomial kernel method}, a well-known method for constructing good approximations of a given function via convolution by a given  polynomial kernel. The difference lies in the selection of the polynomial kernel, which relies on a kernel of the form $(x^Ty)^{2r}$ in \cite{Reznick_1995, Doherty_Wehner_2013} and on a suitably optimized kernel in \cite{Fang_Fawzi_2020}. 
Fang and Fawzi \cite{Fang_Fawzi_2020} use their polynomial kernel to construct a   sos polynomial representation that yields a feasible solution for the program (\ref{eqsosr}), leading to their convergence analysis. 

On the other hand, Doherty and Wehner \cite{Doherty_Wehner_2013} work in the symmetric tensor setting. Additionally, they use their kernel to derive a result on the dual moment side that leads to a `real Quantum de Finetti' type result (see Theorem \ref{theoDW}), which they then exploit to show their convergence analysis for the sos bounds. However, their real QdF theorem applies to real {\em maximally symmetric} matrices, a  symmetry property that is much stronger than being doubly symmetric (see Definition \ref{defmaximallysymmetric}).

\subsection*{Our contributions}
As sketched above, there are similarities and differences between the sos and spectral bounds and their analyses.
This raises several questions that we address in this paper.

\smallskip
A first question is whether the QdF theorem of \cite{Christandl_2007_deFinetti_oneandhalf}  (Theorem \ref{theoQdF}) admits an analog for {\em real} doubly symmetric matrices. A positive answer would immediately imply a 
simple proof of the Doherty-Wehner real QdF (Theorem \ref{theoDW}).  As a first minor contribution, we give a negative answer:   it is not possible to approximate a sequence of real $r$-doubly symmetric matrices (in fact, their partial traces) by real separable ones (see \Cref{prop:nonexistence-of-real-qdf} for the precise statement). We show this using a family of counterexamples inspired by Caves et al. \cite{Caves_Fuchs_Schack_2002}.

\smallskip
A second question that we address in Section \ref{secrealQdF} is whether one can improve the analysis in the real QdF theorem (Theorem \ref{theoDW}) of   \cite{Doherty_Wehner_2013}. We  give a positive answer (see Theorem \ref{theoDWr2}). For this, we make the simple but important observation that 
the polynomial kernel method can also be used to show approximation results on the {\em dual moment side},   by considering {\em adjoint  kernel operators}: we indicate how the kernel operator of Fang and Fawzi \cite{Fang_Fawzi_2020} can be used to strengthen the analysis of Doherty and Wehner \cite{Doherty_Wehner_2013} from $O(1/r)$ to $O(1/r^2)$.  

\smallskip
\MoL{As a third contribution, we further investigate the spectral bounds. We show in Section \ref{sec:no-finite-convergence} that, unlike the moment-sos bounds, they do not enjoy generic finite convergence (Theorem \ref{thm:no-finite-convergence-n}).}
Another natural question is whether the analysis of \cite{Johnston_Lovitz_Vijayaraghavan_2023} in $O(1/r)$ for the spectral parameters can possibly be improved and what is a {\em lower bound} on their exact convergence rate. As our \tcolblue{fourth and main technical} contribution, we show in Section~\ref{secChoiLam} that $p_{\min}-\spec_r(p)=\Omega(1/r^2)$ for  a suitably crafted $5$-variate form of degree 4  (that we construct from the Choi-Lam form, well-known to be nonnegative but not sos). So, there is still a gap between the lower bound $\Omega(1/r^2)$ and the upper bound $O(1/r)$ and it would be very interesting to close this gap. In addition, at this moment, we are not aware of any lower bound for the convergence rate of the sos bounds; showing such a lower bound would also be very interesting.

\smallskip
In addition, we revisit in Section \ref{secbandedQdF} the convergence analysis in $O(1/r)$ for the spectral bounds  from \cite{Johnston_Lovitz_Vijayaraghavan_2023}.
Rather than going via optimization on the complex unit sphere, 
we propose to directly use the QdF theorem to derive a `banded' real QdF theorem (see Theorem~\ref{thm:banded-qdf}). Roughly, 
this result approximates the symmetrized partial trace of a real doubly symmetric positive semidefinite matrix by a {\em scaled} real separable matrix. \tcolblue{So, in a nutshell, the banded real QdF theorem interpolates between the 
QdF theorem (Theorem \ref{theoQdF}, that applies to complex Hermitian, doubly symmetric matrices) and the real QdF (Theorem \ref{theoDW}, that applies to real maximally symmetric matrices). It thus falls within the realm of {\em real} quantum theory, where physical phenomena are modelled using real numbers that to many physicists appear more natural than complex ones. Whether real numbers suffice to model quantum theory is debated. In some settings they do, but  Renou et al. \cite{Renou-Nature} show that different predictions may be obtained when restricting to real numbers in certain (so-called entanglement-swapping) network scenarios.
}

The new banded real QdF theorem (Theorem~\ref{thm:banded-qdf}) can then be directly  used to recover the convergence rate in $O(1/r)$ of \cite{Johnston_Lovitz_Vijayaraghavan_2023} (recalled in Theorem \ref{theoJLV}) for the spectral bounds $\spec_r(p)$, \tcolblue{albeit with other constants than in \cite{Johnston_Lovitz_Vijayaraghavan_2023}. First, in Theorem~\ref{thm:analysis-spectral-hierarchy}, our error analysis involves an explicit constant depending on $n$ and $d$, instead of the unknown parameter $\kappa(M)$ (the condition number of the maximally symmetric matrix representing $s^d$) used in \cite{Johnston_Lovitz_Vijayaraghavan_2023}. Second, in Theorem~\ref{thm:pmax-pmin-spectral}, our analysis   involves the range   $p_{\max}-p_{\min}$ of values that the polynomial $p$ takes on the sphere, instead of other  constants of $p$.} 
\tcolblue{This dependency on the   range of values is   used, e.g., in the error analysis of the moment-sos bounds (see Theorem~\ref{theoFF}), and it is commonly used in continuous optimization (see, e.g., \cite{Vavasis1992,deKlerk_Laurent_Parrilo_2005}). Note that it is well-behaved under  
 translating $p$ by a scalar multiple of $s^d$: The value range does not change while the  constant $\|Q(p)\|_\infty$ (used in \cite[Theorem 4.1]{Johnston_Lovitz_Vijayaraghavan_2023}, see Theorem~\ref{theoJLV}) or $\gamma(Q(p))$ (from Theorem \ref{thm:analysis-spectral-hierarchy}) grows with the scaling factor 
 (see  Appendix \ref{appendix-bad-gamma}).}

\subsection*{Some background on existing literature}

Problem (\ref{eqpmin})  is  computationally hard already for homogeneous polynomials of degree 4. Indeed, it contains some known hard optimization problems, such as the maximum stable set graph problem (due to a reformulation of Motzkin and Straus \cite{Motzkin_Straus}), or deciding convexity of  a quartic homogeneous polynomial, shown to be NP-hard in \cite{AOPT2013}. See, e.g., \cite{deKlerk2008} for more about complexity issues. Huang \cite{Huang2023} shows that the sos hierarchy has finite convergence for generic polynomials.
Polynomial optimization on the unit sphere is relevant to a broad range of applications, e.g., for modelling matrix norms $\|\cdot\|_{2\to p}$, for computing  eigenvalues of  tensors (see, e.g., \cite{Lim}), for capturing entanglement versus separability of states in quantum information theory (see, e.g., \cite{Doherty_Wehner_2013}, \cite{Fang_Fawzi_2020}, \cite{Gribling_Laurent_Steenkamp}, and further references therein).  

\medskip
Prior to the spectral lower bounds $\spec_r(p) \le p_{\min}$ of \cite{Johnston_Lovitz_Vijayaraghavan_2023} for optimization on the sphere,  another spectral approach had  been introduced earlier for minimization  over a general compact set $S$ by  Lasserre \cite{Lasserre2011}.  \tcolblue{He defines  {\em upper bounds} $\ub_r(p)\ge p_{\min}$, obtained by minimizing  the  expected value $\int_Spqd\mu$ of $p$ on $S$, taken over all sos polynomials $q$ of degree at most $2r$ such that $\int_S qd\mu=1$, where $\mu$ is a given measure supported on $S$. 
Computing $\ub_r(g)$ boils down to an eigenvalue computation \cite{Lasserre2011}.
These spectral upper bounds have been analyzed:   a convergence rate in $O(1/r^2)$   for   $\ub_r(p)-p_{\min}$ is shown for the interval $S= [-1,1]$ in \cite{deKlerk-Laurent-2020}
 (by establishing links to extremal roots of orthogonal polynomials), 
for the sphere  in \cite{deKlerk-Laurent-2022}, and  for many more sets $S$  in \cite{Slot-Laurent-2022}. 
Interestingly, there is an intimate connection between the spectral upper bounds and the sos lower bounds, which is (implicitly) exploited in the analysis of the sos lower bounds for the sphere in \cite{Fang_Fawzi_2020} and, more explicitly, for  the boolean cube in \cite{Slot-Laurent-2023} and the ball and the simplex in \cite{Slot2022}.}

\medskip
\MoL{Finally, let us briefly mention some background on (quantum) de Finetti type results.  
De Finetti's theorem is a classical result about infinite sequences of random variables. In its original form \cite{deFinetti1969}, it roughly says the following: Assume the joint probability distribution of an infinite sequence of random variables $ X_1, X_2,\ldots $ is invariant under permutations, in the sense that $ X_{\sigma(1)},X_{\sigma(2)},\ldots $ has the same joint distribution for any bijective map $ \sigma\colon\N\to \N $ (also known as being {\em exchangeable}). Then, the joint distribution of $ X_1, X_2, X_3,\ldots $ is a convex sum of product distributions. }

\MoL{Diaconis and Freedman \cite{Diaconis_Freedman_1980_deFinetti} showed a finite version of this theorem: If   $ X_1,\ldots,X_r $ is an exchangeable sequence of $r$ random variables with joint density $ p_{r}(x_1,\ldots,x_r) $, then, for any $ d\le r $,  the joint density $ p_{d}(x_1,\ldots,x_d) $ of the first $ d $ random variables may be approximated in a suitable norm by a convex combination of product densities $ q(x_1) \cdot \ldots \cdot q(x_d) $, where the $ q $'s are   univariate density functions. The error in the approximation is of the form $  {O}(\frac{1}{r}) $ for constant $d$. See, e.g., \cite[Section 6]{Doherty_Wehner_2013}. }

\MoL{There exist ``quantum'' versions of de Finetti's theorem, both for the infinite and   finite settings. In \Cref{theoQdF}, we will present a finite quantum de Finetti theorem due to Christandl et al. \cite{Christandl_2007_deFinetti_oneandhalf}. It allows to approximate the partial trace of a complex density matrix $ M\colon (\C^n)^{\ot r} \to (\C^n)^{\ot r}  $ with row and column permutation invariance (i.e., double symmetry) by a convex combination of product states.
The approximation error is $  {O}(\frac{1}{r}) $, just as for the Diaconis-Freedman result \cite{Diaconis_Freedman_1980_deFinetti}, and $  {O}(\frac{1}{r}) $ is known to be optimal for both results. An explanation of the philosophical aspects behind the classical and the quantum de Finetti theorem and its significance for quantum information is given by Caves, Fuchs and Schack \cite{Caves_Fuchs_Schack_2002}. The paper by Caves, Fuchs and Schack also serves as a friendly introduction to de Finetti theorems in quantum information. 
} 

\MoL{Finally, let us mention the work of Navascu\'es, Owari and Plenio \cite{Navascues_2009_deFinetti_ppt} who show a finite quantum de Finetti result with an improved rate in $O(1/r^2)$ for quantum bipartite states with a positive semidefinite partial transpose (aka the PPT condition).}

\MoL{We will mention in  Section \ref{secfinal} some close connections between de Finetti type results and  the classical moment problem.
}

\subsection*{Organization of the paper}
In Section \ref{sec:prelims} we gather definitions and preliminaries needed for the paper, about polynomials, symmetric tensors, linear functionals, quantum de Finetti theorems, and harmonic polynomials. In Section \ref{secrealQdF} we show  an improved real quantum de Finetti theorem for maximally symmetric matrices. 
We present a new `banded' analog for real doubly symmetric matrices in Section~\ref{secbandedQdF}, which we use to provide   alternative analyses of the spectral bounds  in Section~\ref{secJLV}. In the same section, we also prove a lower bound on the convergence rate of the spectral hierarchy \MoL{and we show that the spectral hierarchy does not enjoy generic finite convergence.}  We close the paper with some discussion in Section \ref{secfinal}.

\section{Preliminaries}\label{sec:prelims}

We begin with some notation that is used throughout the paper. $\N=\{0,1,2,\ldots\}$ is the set of nonnegative integers. For an integer $d\in \N$, $\N^n_d$ denotes the set of sequences $\alpha\in \N^n$ with $|\alpha|:=\alpha_1+\ldots+\alpha_n=d$. We let $\MFS_d$ denote the set of permutations of the $d$-element set $[d]=\{1,\ldots,d\}$. We let $e_1,\ldots,e_n$ denote the standard unit vectors in $\C^n$, where all entries of $e_i$ are 0 except 1 at the $i$th coordinate. For a vector $u\in \C^n$,  $\|u \|_2=\sqrt {u^*u}$ denote its Euclidean norm. The unit sphere in $\R^n$ is  denoted $\oS^{n-1}=\{x\in\R^n: \|x\|_2=\sum_{i=1}^nx_i^2=1\}$.  We let $ \conv(S) $ denote   the convex  hull of a set $ S $ and  $ \cone(S) $ is its conic hull.  

For $\K=\R$ or $\C$, we use the trace inner product on the space $\K^{n\times n}$ of matrices:  $\langle A,B\rangle= \Tr(A^*B)$ for $A,B\in \K^{n\times n}$. Here, $A^*= \bar A^T$ denotes the complex conjugate transpose of $A$; $A$ is Hermitian if $A^*=A$. For a Hermitian matrix $A$, $\lambda_{\max}(A)$ and $\lambda_{\min}(A)$ denote, respectively, its largest and smallest eigenvalues.  The matrix $A$ is positive semidefinite, written as $A\succeq 0$, if $\lambda_{\min}(A)\ge 0$. For Hermitian matrices $ A, B $, we write $ A\preceq B $ if $ B-A $ is positive semidefinite. \tcolblue{Throughout, $I_n$ denotes the $n\times n$ identity matrix.}

Let $A$ be an $n\times n$ Hermitian matrix with eigenvalues $ \lambda_1,\ldots,\lambda_n $. Its Schatten-$ p $ norm is $ \|A\|_{p} = \sqrt[p]{\sum_{i = 1}^{n} |\lambda_i|^{p}} $. Hence, $\|A\|_1=\Tr(A)$ if $A\succeq 0$. The dual norm of $ \|A\|_{1} $ is the Schatten $\infty$-norm $\|A\|_\infty$, which is the largest absolute value of an eigenvalue, i.e., $\|A\|_\infty=\max\{\lambda_{\max}(A), -\lambda_{\min}(A)\}$. Note also that the Schatten 2-norm coincides with the Frobenius norm, so $\|A\|_2 = \sqrt{\langle A,A\rangle}$. We record the following relation for later use: if $A$ has $r$ nonzero eigenvalues, then
\begin{align}\label{eqnorm12}
\|A\|_1\le \sqrt r \|A\|_2\le \sqrt r\|A\|_1.
\end{align}

\MoL{In what follows we group definitions and preliminary facts needed for this paper about tensors, (homogeneous) polynomials and linear functionals on polynomials, Bose symmetric and maximally symmetric matrices, quantum de Finetti type results, and harmonic polynomials.  These are somewhat scattered  notions, so we give pointers to the literature throughout the text. In addition, for general background, we refer, e.g., to  the survey about tensors by Kolda and Bader \cite{Kolda-Bader-2009} and to the monographs by Lasserre  \cite{Lasserre-2015} or Nie \cite{Nie-2023} about polynomials and moment theory, by Nielsen and Chuang \cite{Nielsen-Chuang} or Watrous \cite{Watrous_2018} about quantum information theory, and  by M\"uller \cite{Mueller-1966} about harmonic polynomials.}

\subsection{Symmetric tensors,  doubly symmetric matrices and  maximally symmetric matrices}\label{secsymmetry}

Let $\K=\R$ or $\C$ be the real or complex field. For integers $r,n\ge 1$, $(\K^n)^{\ot r}$ denotes the space of $r$-tensors on $\K^n$, with elements $v=(v_{\ui})_{\ui \in [n]^r}$. Here, we use the notation $\ui=(i_1,\ldots,i_r)\in [n]^r$ (or simply $i_1\ldots i_r$) for an $r$-sequence in $[n]^r$. Given two integers $1\le d\le r$, it may also be convenient to denote  $r$-sequences as
the concatenation $\uj \uk $ of a $d$-sequence $\uj\in [n]^d$ and an $(r-d)$-sequence $\uk\in [n]^{r-d}$.
\tcolblue{For  vectors $u_i\in \K^{n_i}$, $i=1,2$, their tensor product is the vector $u_1\ot u_2\in\K^{n_1}\ot \K^{n_2}$ with entries
$(u_1\ot u_2)_{i_1i_2}=(u_ 1)_{i_1}(u_2)_{i_2}$ for $(i_1,i_2)\in[n_1]\times [n_2]$. For $u\in\K^n$, $u^{\ot r} = u\ot \cdots \ot u$ (taking the tensor product of $r$ copies).}

 We now recall several notions of symmetry for tensors. \tcolblue{In Appendix \ref{appendix-example} we illustrate some   of these notions on examples.}
 
Permutations $\sigma\in \MFS_r$ act on $[n]^r$ and thus on $r$-tensors, in the following way.
For $\ui\in [n]^r$, set $\ui^\sigma=(i_{\sigma(1)},\ldots, i_{\sigma(r)})$ and, for  $v\in (\K^n)^{\ot r}$, define $v^\sigma$ as the $r$-tensor with entries $(v^\sigma)_{\ui}= v_{\ui^\sigma}$ for $\ui\in [n]^r$.
Then, $v$ is said to be a {\em symmetric tensor} if $v^\sigma=v$ for all $\sigma\in \MFS_r$, and $S^r(\K^n)$ denotes the space of symmetric $r$-tensors on $\K^n$. We let $\Pi_r$ denote the projection from the space $(\K^n)^{\ot r}$ to the symmetric subspace $S^r(\K^n)$. Concretely, we have 
\begin{align}\label{eqPir}
\Pi_r=\frac{1}{r!}\sum_{\sigma\in \MFS_r} P_\sigma,
\end{align}
where $P_\sigma$ is the square matrix indexed by $[n]^r\times [n]^r$, with entries 
$(P_\sigma)_{\ui, \uj}=1$ if $\uj=\ui^\sigma$ and 0 otherwise, for $\ui,\uj\in [n]^r$.

Consider a  matrix $M$ indexed by the set $[n]^r\times [n]^r$, so that $M$ represents an endomorphism from $(\K^n)^{\ot r}$ to $(\K^n)^{\ot r}$. Then, any permutation $\sigma\in \MFS_r$  induces naturally a permutation of  the columns (and rows) of $M$. 

\begin{definition}[Doubly symmetric matrix]\label{defdoublesymmetric}
A matrix $M\in \K^{[n]^r\times [n]^r}$  is called  {\em $r$-column symmetric} (resp., {\em $r$-row symmetric}) if $M\Pi_r =M$ (resp., $\Pi_rM=M$). If $M$ is complex Hermitian or real symmetric, then $M$ is $r$-column symmetric if and only if it is $r$-row symmetric, or equivalently if $\Pi_rM\Pi_r=M$, in which case  $M$ is called {\em $ r $-doubly symmetric} (or just {\em doubly symmetric} if $r$ is clear from the context).
\end{definition}

Hermitian doubly symmetric matrices are also known as {\em Bose symmetric} matrices in the quantum information literature. \MoL{The following stronger symmetry property is sometimes needed (it is considered, e.g.,   in \cite{Doherty_Wehner_2013}).}

\begin{definition}[Maximally symmetric matrix]\label{defmaximallysymmetric}
A  matrix $M\in \K^{[n]^r\times [n]^r}$ can be   viewed as  a $2r$-tensor $\vvec(M)\in (\K^n)^{\ot 2r}$, with entries  $\vvec(M)_{\ui\uj} = M_{\ui,\uj}$ for $\ui,\uj\in [n]^r$.
Then, $M$ is said to be {\em maximally symmetric} when $\vvec(M)$ is a symmetric $2r$-tensor.
\end{definition}

\tcolred{Note that a Hermitian matrix that is maximally symmetric is in fact real-valued (see also Lemma \ref{lem:ppt-sym-is-real-rank-2} below).} Let $\PiMS$ denote the projection from the space $\K^{[n]^r\times [n]^r}$ onto $\MaxSym_r(\K^n)$, the  space of maximally symmetric matrices. Concretely, the projection $\PiMS(M)$ is obtained by first constructing the symmetric $2r$-tensor $\Pi_{2r}(\vvec(M))$ and then reshaping it as a square matrix indexed by $[n]^r$ (by splitting the $2r$ indices into two classes of size $r$, which can be done   arbitrarily  as the tensor is $2r$-symmetric).

\begin{example}\label{ex:IdMS}
As an illustration, for $n=r=2$, we have
\begin{align}\label{eqmaxsymI2}
\PiMS(I_2^{\ot 2})=\left(\begin{matrix} 
1 & 0 & 0 & 1/3 \cr
0 & 1/3 & 1/3 & 0\cr
0 & 1/3 & 1/3 & 0 \cr
1/3 & 0 & 0 & 1
\end{matrix}\right),
\end{align}
ordering the index set $[2]^2$ as $11, 12, 21, 22$ (corresponding to 
the standard basis elements ordered as $e_1\ot e_1, e_1\ot e_2, e_2\ot e_1, e_2\ot e_2$).
\tcolblue{See Appendix \ref{appendix-example} for details.}
\end{example}

\begin{example}\label{ex:a-b-expansion}
	Given two vectors $ a, b\in \R^n $, one can check that
	\begin{align*}
		&\PiMS((aa^{T} + bb^{T})^{\otimes 2}) = aa^{T}\ot aa^{T} + bb^{T}\ot bb^{T} +  \\
		&\frac{1}{3}(aa^{T} \ot bb^{T} + ab^{T}\ot ab^{T} + ab^{T}\ot ba^{T} + ba^{T}\ot ab^{T} + ba^{T}\ot ba^{T} + bb^{T}\ot aa^{T})\nonumber
	\end{align*}
\end{example}

We now examine what happens if we project an element of the form $ (uu^{\ast})^{\otimes d} $ to the maximally symmetric subspace. 
The next result is shown in \cite{Johnston_Lovitz_Vijayaraghavan_2023}, but we give a proof in order to provide some insight.
\begin{lemma}\label{lem:ppt-sym-is-real-rank-2}\cite{Johnston_Lovitz_Vijayaraghavan_2023}
	Let $ u\in \C^n $. Write $ u = a + \bfi b $, with $ a,b\in \R^n $. Then, 
	\begin{align*}
		\PiMS((uu^{\ast})^{\otimes d}) = \PiMS ((a a^{T} + b b^{T})^{\otimes d}). 
	\end{align*}
	In particular, the maximal symmetrization of $ (uu^{\ast})^{\otimes d} $ is real valued for complex $ u $. 
\end{lemma}
\noindent The expansion of $ \PiMS ((a a^{T} + b b^{T})^{\otimes d}) $   for $ d = 2 $ is given in  \Cref{ex:a-b-expansion}. 
As a tool to show Lemma \ref{lem:ppt-sym-is-real-rank-2}, we use the notion of {\em partial transpose}.
For a matrix $M$ indexed by $[n]^d\times [n]^d$, taking the partial transpose w.r.t. the first register yields the matrix $\tilde M$ with entries $\tilde M_{i_1i_2\ldots i_d, j_1j_2\ldots j_d}=M_{j_1i_2\ldots i_d, i_1j_2\ldots j_d}$ for $i_1,j_1,\ldots, i_d,j_d\in [n]$. This corresponds to applying to $\vvec(M)$ the permutation $\sigma \in \mathfrak{S}_{2d} $ that flips $1$ and $d+1$. 
Taking the partial transpose w.r.t. to any subset of the $d$ registers is defined analogously. Note $\PiMS(M)=\PiMS(\tilde M)$ for any partial transpose $\tilde M$ of $M$.

\begin{proof} [Proof of Lemma \ref{lem:ppt-sym-is-real-rank-2}]
Note that $ uu^*+ (uu^*)^T= uu^* + \bar{u} \bar{u}^*= 2aa^T + 2bb^T$. Hence, taking the average of 	$ (uu^{\ast})^{\otimes d} $ and its first partial transpose gives the matrix
\begin{align*}
		\frac{1}{2}(uu^{\ast})^{\otimes d} + \frac{1}{2}(\bar{u}\bar{u}^{\ast})\otimes (uu^{\ast})^{\otimes d-1} 
		& = 
		\frac{1}{2}(uu^*+ \bar{u}\bar {u}^*) \ot (uu^*)^{\ot (d-1)}\\
		&=
		(aa^{T} + bb^{T}) \otimes (uu^{\ast})^{\otimes (d-1)}. 
\end{align*}
Taking iteratively the average of the partial transposes at all other registers  yields the matrix $ (aa^T + bb^T)^{\ot d}$.
Hence, $ \PiMS((uu^{\ast})^{\otimes d}) = \PiMS ((a a^{T} + b b^{T})^{\otimes d}) $ by the above observation.
\end{proof}

We will also use the notion of {\em partial trace}.
\begin{definition}[Partial trace]
Consider a  matrix $M$ indexed by $[n]^r\times [n]^r$ and an integer $1\le d\le r$.
Then, the  {\em partial trace} $\Tr_{r-d}(M)$ is the  matrix indexed by $[n]^d\times [n]^d$ obtained by `tracing out the last $r-d$ registers' of $(\K^n)^{\ot r}$, i.e., with entries 
$(\Tr_{r-d}(M))_{\ui,\uj}=\sum_{\uk\in [n]^{r-d}} M_{\ui\uk, \uj\uk}$ for $\ui,\uj\in [n]^d$. In other words,
\begin{align*}
\langle M, A\otimes I_n^{\ot (r-d)}\rangle =\langle \Tr_{r-d}(M), A\rangle\ \text{ for any } A\in \K^{[n]^d\times [n]^d}.
\end{align*}
\end{definition}

\subsection{Polynomials and linear functionals on polynomials}

Throughout, let $ R := \R[x] = \R[x_1,\ldots,x_n] $ denote  the polynomial ring in variables $ x=(x_1,\ldots,x_n) $. For an integer $d\in\N$,  $ R_d $ denotes its $ d $-th graded component, which consists of the homogeneous polynomials (aka forms) with degree $d$, and $R_{\le d}=R_0\oplus R_1\oplus\ldots\oplus R_d$ consists of the polynomials with degree at most $d$. 
We let $[x]_d= (x^\alpha)_{\alpha\in \N^n_d}$ denote the vector of all monomials with degree $d$.
Throughout we use the notation  $$s=x_1^2+\ldots +x_n^2.$$

\begin{definition}[Representing matrix of a polynomial]\label{defGramindex}
A symmetric matrix $G$ indexed by $\N^n_d$ is said to be a {\em representing matrix} of a polynomial $f\in R_{2d}$ if the polynomial identity $f=[x]_d^T G [x]_d$ holds. 
\end{definition}

A polynomial $f\in R$ is a {\em sum of squares} (abbreviated as sos) if $ f = p_1^2  + \ldots + p_N^2 $ for some $ N\in \N $ and $ p_1,\ldots,p_N\in R$. Then, all $p_i$ lie in $R_d$ (resp., $R_{\le d}$) if $f\in R_{2d}$ (resp., $f\in R_{\le 2d}$). The following characterization is well-known \tcolblue{(see \cite{Choi-Lam-Reznick-1995,Powers-Wormann-1998})}.

\begin{lemma}\label{lem:sos-equiv-gram-matrix}
A polynomial $f$ is a sum of squares if and only if it admits a positive semidefinite representing matrix (then, often called a {\em Gram matrix} of $f$).
\end{lemma}

Let $\Sigma$ denote the cone of sos polynomials and set $\Sigma_d=\Sigma\cap R_{2d}$. 
Let $\MP$ denote the cone of globally nonnegative polynomials and set $\MP_d=\MP\cap R_{2d}$. 
So, $\Sigma_d\subseteq \MP_d$ are convex cones in $R_{2d}$.  

Let $\MP(\oS^{n-1})$ denote the cone of polynomials that are nonnegative on the sphere $\oS^{n-1}$ and set $\MP_d(\oS^{n-1}) =\MP(\oS^{n-1})\cap R_{2d}$. \tcolblue{Note that $\MP_d(\oS^{n-1})=\MP_d$ (since nonnegativity on the sphere implies global nonnegativity for a homogeneous polynomial).}

\medskip
The dual cones $\Sigma^*_d$ and $\MP^*_d$ live in the dual vector space $R^*_{2d}$, which consists of the linear functionals $L:R_{2d}\to \R$. 
The following notion of moment matrix is useful to characterize  membership in the dual   cones. 

\begin{definition}[Moment matrix]\label{defmomentindex}
The moment matrix of $L\in R_{2d}^*$ is the matrix $M_d(L)=L([x]_d[x]_d^T)$, where $L$ is applied entry-wise.
\end{definition}

\tcolblue{We refer to Lemma \ref{lemsos} and Proposition \ref{propsosP} below for a characterization of the dual cones 
$\Sigma_d^*$ and $\MP_d^*$  in terms of properties of moment matrices.}

\subsection{Homogeneous polynomials and symmetric tensors}

Let $ f\in R_{d} $ be  homogeneous  of degree $ d $ in $ n $ variables $ x_1,\ldots,x_n $. It is most common to write 
\begin{align}\label{multi-index} 
f = \sum_{\alpha\in \N_{d}^n} f_{\alpha} x^{\alpha} =\langle (f_\alpha)_{\alpha\in \N^n_d}, [x]_d\rangle
\end{align} 
in the (usual) monomial basis, for some scalars $f_\alpha$ ($\alpha\in \N^n_d$);
we call this \emph{multiplicity indexing} or \emph{multi-indexing} of $f$. 
\tcolblue{For a sequence $\ui=(i_1,\ldots,i_d)\in [n]^d$, set $x^{\ui}=x_{i_1}\cdots x_{i_d}$. Then, }
one can alternatively write 
\begin{align}\label{tensor-index}
f =\sum_{\ui\in [n]^d} v_{\ui}x^{\ui}=\langle v, x^{\ot d}\rangle,
\end{align}
for some (uniquely determined) symmetric $d$-tensor  $v\in S^d(\R^n)$; we call this {\em tensor-indexing} of $f$. To see this, for $\ui=(i_1,\ldots,i_d)\in[n]^d$, define $\alpha(\ui)\in \N^n_d$ with $\alpha(\ui)_k$ being the number of occurrences of the symbol $k\in[n]$ in the multiset $\{i_1,\ldots,i_d\}$. Then, $x^{\ui} = x^{\alpha(\ui)}$ and $v_{\ui}=v_{\uj}$ if $\alpha(\ui)=\alpha(\uj)$ (since $v$ is a symmetric tensor). \tcolblue{For a given $\beta\in \N^n$,} the number of sequences $\ui\in [n]^d$ with $\alpha(\ui)=\tcolblue \beta$  equals the multinomial coefficient ${d\choose \tcolblue \beta}={d!\over \tcolblue{\beta_1!\cdots \beta_n!}}$. Comparing (\ref{multi-index}) and (\ref{tensor-index}) gives  $f_{\tcolblue{\beta}}={d\choose \tcolblue \beta} v_{\ui}$ if $\alpha(\ui)=\tcolblue \beta$.

As a consequence, moment matrices and representing  matrices of polynomials may also be defined in tensor-indexing, as analogs of Definitions \ref{defGramindex} and \ref{defmomentindex}.

\begin{definition} \label{defGramtensor}
	Let $f\in R_{2d}$. A symmetric matrix $G$ indexed by $[n]^d$ is called a {\em tensor-indexed representing matrix} of  $ f $ if the following polynomial identity holds:
	\begin{align*}
		f = (x^{\otimes d})^{T}G x^{\otimes d} = \langle G, (x x^{T})^{\otimes d} \rangle. 
	\end{align*}
\end{definition}
	
	Then, any polynomial  $f\in R_{2d}$ admits  a (unique) tensor-indexed representing matrix  that is maximally symmetric, denoted by $\MaxSym(f)$. Namely, we have $\MaxSym(f)=\PiMS(G)$, where $G$ is any tensor-indexed representing matrix of $f$.
Moreover,
in analogy with \Cref{lem:sos-equiv-gram-matrix}, a polynomial $ f $ is a sum of squares if and only if it has a tensor-indexed positive semidefinite representing matrix.

\begin{definition}\label{defmomenttensor}
	Let $ L \in R_{2d}^{*}$. Then, the {\em tensor-indexed  moment matrix} of $L$ is  
	\begin{align*}
		L(x^{\ot d}(x^{\ot d})^T)= (L(x^{\ui} x^{\uj}))_{\ui,\uj\in [n]^d}.
	\end{align*}
\end{definition}

Note that the tensor-indexed moment matrix $L(x^{\ot d}(x^{\ot d})^T)$ of $L\in R_{2d}^*$ arises from its multi-indexed moment matrix $M_d(L)=L([x]_d([x]_d)^T)$ by repetitions of its rows/columns (with two rows/columns indexed by $\ui,\uj\in [n]^d$ being equal when $\alpha(\ui)=\alpha(\uj)$).
It is easy but important to realize that maximally symmetric matrices and tensor-indexed moment matrices  are in fact equivalent notions (see \cite{Doherty_Wehner_2013}).

\begin{lemma}\label{lem:moment-matrix-equiv-maxsym}
	Let $M$ be a real matrix indexed by $[n]^d$. 
	The following are equivalent: 
	\begin{enumerate}
		\item[(i)] $ M $ is the tensor-indexed moment matrix of some linear functional $L\in R_{2d}^{\tcolblue\ast}$. 
		\item[(ii)] $ M $ is maximally symmetric. 
	\end{enumerate}
\end{lemma}
\begin{proof}
(i) $\implies$ (ii) follows since $x^{\ot d}(x^{\ot d})^T$ is maximally symmetric, \tcolblue{which implies  $M=L(x^{\ot d}(x^{\ot d})^T)$ is  maximally symmetric.}
Conversely, assume $M$ is maximally symmetric. Then, one can define $L\in R_{2d}^*$ as follows. For $\gamma\in \N^n_{2d}$, pick $\uk\in [n]^{2d}$ such that $\alpha(\uk)=\gamma$, i.e., $x^\gamma=x^{\uk}$. Then, write $\uk=\ui\uj$ as the concatenation of two sequences $\ui,\uj\in [n]^d$ and set $L(x^\gamma)=M_{\ui,\uj}$. Note that this   definition  does not depend on the choice of $\uk$  since $M$ is maximally symmetric. By construction, $M$ is the tensor-indexed moment matrix of $L$, as desired.
\end{proof}

\tcolblue{We now characterize the dual cones $\Sigma_d^*$ and $\MP_d^*=\MP_d(\mathbb S^{n-1})^*$, both in terms of the (usual) moment matrices $M_d(L)$ and their tensor analogs $L(x^{\ot d}(x^{\ot d})^T)$.
The following is  well-known and easy to see. }

\begin{lemma}\label{lemsos}
For $L\in R_{2d}^*$, $L\in \Sigma_d^{\ast} \Longleftrightarrow M_d(L)\succeq 0 \Longleftrightarrow L(x^{\ot d} (x^{\ot d})^{T}) \succeq 0.$
\end{lemma}

\tcolblue{Given $L\in R_{2d}^*$, recall that a (Borel positive) measure $\mu$ is called a {\em representing measure} of $L$ if $L(x^\alpha)=\int x^\alpha d\mu(x)$ holds for all $\alpha\in\oN^n_{2d}$. Clearly, $L\in \MP_{d}^*$ if it has a representing measure.
Classical results of Haviland \cite{Haviland1936} and Tchakaloff \cite{Tchakaloff1957} show  that  linear functionals that are nonnegative on nonnegative polynomials of a given degree 
have a  representing measure that can be chosen to be finite atomic.
 We need the following  variant of these results for linear functionals on \emph{homogeneous} polynomials, which may be found, e.g., in \cite[Lemma 4.2]{Riener_Schweighofer_2018}. }

\begin{proposition}\label{propsosP}
	Let $ L\in R_{2d}^{\ast}$. The following assertions are equivalent.
	\begin{itemize}
	\item[(i)] $L\in \mathcal{P}_d^{\ast}$ (equivalently, $L\in \MP(\mathbb S^{n-1})^*$).
	\item[(ii)] $M_d(L)\in \cone\{[u]_d[u]_d^T: u\in \R^n, \|u\|=1\}.$
	\item[(iii)] 
		$L(x^{\ot d} (x^{\ot d})^{T}) \in \cone \{u^{\ot d} (u^{\ot d})^{T} : u\in \R^n \}.$
	\end{itemize}
\end{proposition}

Throughout, we will use the norm $\|\cdot \|_\infty$ on the polynomial space $R_{d}$, its dual norm $\|\cdot\|_1$ on the  dual space $R^*_d$, and (following  \cite{Doherty_Wehner_2013}) the associated norm $\|\cdot\|_{F1}$ for maximally symmetric matrices, defined as follows.

\begin{definition}\label{defnorm}
For $f\in R_d$, $\|f\|_\infty=\max\{|f(x)|: x\in \oS^{n-1}\}$ and, for $L\in R_d^*$,  $\|L\|_1=\max\{L(f): f\in R_d, \|f\|_\infty=1\}$. For a maximally symmetric matrix $M$ indexed by $[n]^d$,  we set $\|M\|_{F1}=\|L\|_1$, where $L\in R_{2d}^*$ is the linear functional whose moment matrix is $M$, i.e., such that $L(x^{\ot d}(x^{\ot d})^T)=M$.
\end{definition}

\begin{lemma}\label{lemnorm}
	If $f\in R_{2d}$ and  $G$ is a (tensor-indexed) matrix representing $f$, then 
	\begin{align}\label{eqnormfG}
		\|f\|_\infty \le \|G\|_\infty.
	\end{align}
	If $M$ is a real maximally symmetric matrix indexed by $[n]^d$, then
	\begin{align}\label{eqnormLM}
		\|M\|_1\le \|M\|_{F1}.
	\end{align}
\tcolblue{Here, $\|G\|_\infty$ and $\|M\|_1$ are the Schatten norms.}
\end{lemma}

\begin{proof}
We first show (\ref{eqnormfG}). By definition,  we have $f(x)=\langle G, (xx^T)^{\ot d}\rangle$. Hence, for any unit vector $x\in\R^n$, $|f(x)|\le \max\{\lambda_{\max}(G),-\lambda_{\min}(G)\}=\|G\|_{\infty}$, showing (\ref{eqnormfG}). Now, we show (\ref{eqnormLM}). By duality of the Schatten 1- and $\infty$-norms, $\|M\|_1=\langle M,G\rangle$, for some matrix   $G$ with $\|G\|_\infty=1$. Let $L\in R_{2d}^*$ be the linear functional with tensor-indexed moment matrix $M$, i.e., 
$M =L((xx^T)^{\ot d})$, and consider the polynomial $f(x)=\langle G, (xx^T)^{\ot d}\rangle$, \tcolblue{with representing matrix $G$. Then, we have
$$\langle M,G\rangle =\langle L((xx^T)^{\ot d}), G\rangle =L( \langle (xx^T)^{\ot d}, G\rangle)= L(f).$$}  Hence, $\|M\|_1= \langle M,G\rangle =L(f) \le \|L\|_1\|f\|_\infty\le \|M\|_{F1} \|G\|_\infty=\|M\|_{F1}$, where in the last inequality we use $\|L\|_1=\|M\|_{F1}$ and  (\ref{eqnormfG}). So, (\ref{eqnormLM}) holds.
\end{proof}

\subsection{Quantum de Finetti theorems}\label{secQdF}

We now have all tools in hand to introduce the Quantum de Finetti type theorems that are relevant to our treatment.

We first present the  Quantum de Finetti theorem, due to Christandl, K\"onig, Mitchison and Renner \cite{Christandl_2007_deFinetti_oneandhalf}; we also refer to the exposition by  Watrous \cite[Theorem~7.26]{Watrous_2018} and  K\"onig and Renner \cite{Koenig_Renner_2005}. We first state it for matrices   $M=vv^*$ with $v\in S^r(\C^n)$, as is customary in the literature. The result easily extends to the case when $M$ is a  positive semidefinite $r$-doubly symmetric Hermitian matrix. 

\begin{theorem}[Christandl et al. \cite{Christandl_2007_deFinetti_oneandhalf}, Quantum de Finetti,  abbreviated as QdF]\label{theoQdF}
Let $v\in S^r(\C^n)$ be a unit vector. Then, for any integer $1\le d\le r$, there exists 
\begin{align}\label{eqtauCd}
\tau\in \mathcal{C}_d(\C^n)=\conv( u^{\ot d} (u^{\ot d})^*: u\in \C^n, \|u\|_2=1)
\end{align}
such that 
\begin{align}\label{eqQdFCv}
\|\Tr_{r-d}(vv^*)-\tau\|_1 \le {4d(n-1)\over r+1}.
\end{align}
\end{theorem}

\begin{corollary}[Quantum de Finetti, for doubly symmetric Hermitian matrices]\label{corQdF}
Let $M$ be a Hermitian matrix indexed by $[n]^r$. Assume $M$ is    doubly symmetric,  positive semidefinite and $\Tr (M)=1$. Then, for any integer $1\le d\le r$, there exists $\tau\in \mathcal{C}_d(\C^n)$ (as defined in (\ref{eqtauCd})) 
such that 
\begin{align}\label{eqQdFCM}
\|\Tr_{r-d}(M)-\tau\|_1 \le {4d(n-1)\over r+1}.
\end{align}
\end{corollary}

\begin{proof}
Consider the spectral decomposition $M=\sum_l \lambda_l v_lv_l^*$, where the eigenvalues $\lambda_l$ are nonnegative and the eigenvectors $v_l$ are unit vectors in  $S^r(\C^n)$ (since $M$ is doubly symmetric). 
By Theorem \ref{theoQdF}, for each $l$  there exists  $\tau_l\in \mathcal{C}_d(\C^n)$ that satisfies (\ref{eqQdFCv}) w.r.t. $v_l$. For  $\tau:=\sum_l \lambda_l \tau_l$ one can easily see that (\ref{eqQdFCM}) holds, since $ \|\Tr_{r-d}(M)-\tau\|_1 \le (\sum_{l} \lambda_l)    \max_l \|\Tr_{r-d}(v_lv_l^*) - \tau_l\|_{1} $ and $ \sum_l \lambda_l = \Tr(M) = 1$. 
\end{proof}

The above results are stated for complex Hermitian matrices. One may wonder whether they extend to real doubly symmetric matrices, now achieving membership in the set 
\begin{align}\label{eq:Cd-definition}
	\mathcal{C}_d(\R^n) = \conv(u^{\ot d}(u^{\otimes d})^T \mid u\in \R^n, \|u\|_2 =1 ),
\end{align}
\tcolblue{(compare with the cone in Proposition \ref{propsosP}(iii)),} instead of the set $\mathcal{C}_d(\C^n)$ from (\ref{eqtauCd}). The answer is negative, as we will show later in \Cref{prop:nonexistence-of-real-qdf} by giving a family of counterexamples.

On the other hand, a real analog of Corollary \ref{corQdF} has been shown by Doherty and Wehner \cite{Doherty_Wehner_2013}  when making a stronger assumption on the matrix $M$, namely assuming  $M$ is maximally symmetric instead of just doubly symmetric. 

\begin{theorem}[Doherty and Wehner \cite{Doherty_Wehner_2013}, Real QdF, for maximally symmetric matrices] 
\label{theoDW}
Let $ n,d,r\in \N $ with $ n\ge 3$ and $r>d$. Let $M$ be a real symmetric matrix indexed by $[n]^r$. Assume $M$ is maximally symmetric, positive semidefinite and $\Tr(M)=1$.
Then, there exists a matrix $ \tau\in\mathcal{C}_d(\R^n) $  such that 
\begin{align*}
	\|\Tr_{r-d} (M) - \tau\|_{F1} \le \frac{c_{n, d}}{r}, 
\end{align*}
where $ c_{n, d} $ is a constant  depending on $ n, d $, and $\|\cdot \|_{F1}$ is defined in Definition~\ref{defnorm}.
\end{theorem}

\tcolblue{Finally, we recall some useful facts about harmonic polynomials, that we will use later   to show an improved analysis  for the above theorem (see Theorem \ref{theoDWr2}).}

\subsection{Harmonic polynomials}\label{secharmonic}

A polynomial $f\in R$ is called {\em harmonic} if it vanishes under the Laplace operator, i.e., if $\Delta f:=\sum_{i=1}^n {\partial^2f \over \partial x_i^2}=0$. For $r\ge 1$, let $H_r=\{f\in R_r: \Delta f=0\}$ denote the set of harmonic polynomials of degree $r$. Then, the space $R_r$ of homogeneous polynomials of degree $r$ decomposes as the direct sum $R_r=H_r\oplus sR_{r-2}$, where $ s = x_1^2 + \ldots + x_n^2 $ \MoL{(see Remark \ref{remharmonic})}. In particular, this gives the decomposition of $R_{2r}$ into harmonic subspaces:
\begin{align}\label{eqdecharmonic}
	R_{2r} = H_{2r} \oplus sH_{2r-2} \oplus s^2H_{2r-4} \oplus \ldots \oplus s^{r}H_{0}.
\end{align}
This decomposition arises from the irreducible $O(\R^n)$-submodules of the polynomial space.
Concretely, this means that every $ f\in R_{2r} $ has a unique decomposition
\begin{align}\label{eqdecharmonic-homogeneous-with-s}
	f= \sum_{k=0}^{r} s^{r-k} f_{2k},
\end{align}
where $ f_{2k}\in H_{2k} $ is harmonic of degree $2k$. As the polynomial 
$ s $ equals the constant $ 1 $ on the sphere, many authors prefer to write the harmonic decomposition as
\begin{align}\label{eqdecharmonic-homogeneous-without-s}
	f= \sum_{k=0}^{r} f_{2k} \text{ on } \oS^{n-1}.
\end{align}

\begin{remark}\label{remharmonic}
	For the irreducibility of the spaces $ H_r $ as $ O(\R^n) $-modules, we refer to \cite[Section 2.4]{Doherty_Wehner_2013} and the references therein. 
For the reader's convenience	we give a brief argument for the identity $ R_r = H_r \oplus sR_{r-2}$, the decompositions \eqref{eqdecharmonic} and \eqref{eqdecharmonic-homogeneous-with-s} are then an immediate consequence. Observe that $ R_{r} $ may be endowed with the well-known \emph{apolar inner product} (see, e.g., \cite[eq. (8)]{Aldaz_Render_2023}):  $ \langle f, g \rangle = f(\partial_1,\ldots,\partial_n) g $ for $f,g\in R_r$, where $ f(\partial_1,\ldots,\partial_n) $ is a polynomial expression in partial derivatives.  It follows directly from the definition that $ \langle f, sg \rangle = \langle \Delta f, g \rangle $ for any $ f\in R_{r}$ and $ g\in R_{r-2} $. Hence, $ H_r $ is the orthogonal complement of $ sR_{r-2} $ \ATB{and both spaces are $ \mathcal{O}(\R^n) $-invariant.} 
\end{remark}
\section{Improved  Real QdF  theorem for maximally symmetric matrices} \label{secrealQdF}

In this section, we revisit the real QdF theorem (Theorem \ref{theoDW}) of Doherty and Wehner  \cite{Doherty_Wehner_2013} for maximally symmetric matrices, and show an improved convergence rate in $O(1/r^2)$ instead of $O(1/r)$. For this, we use the analysis of the sos bounds by Fang and Fawzi \cite{Fang_Fawzi_2020}, who design a tailored linear operator $ \bfK_r $, which approximates nonnegative polynomials on the sphere by sums of squares of degree $ 2r $. This linear operator $ \bfK_r $ is constructed from a  polynomial integration kernel. We go one step further and use the adjoint polynomial kernel operator to show the following result.

\begin{theorem}\label{theoDWr2}
Let $M$ be a real symmetric matrix indexed by $[n]^r$ and consider an integer  $d\in \{1,\ldots,r\}$. Assume $M$ is maximally symmetric, positive semidefinite and $\Tr(M)=1$. Then, there exists $\tau\in \mathcal C_d(\R^n)$ 
such that 
\begin{align}
	\|\Tr_{r-d}(M)- \tau\|_{1} \le \|\Tr_{r-d}(M)- \tau\|_{F1}\leq c_d {n^2\over r^2}  \text{ for all } r\ge c'_d n,
	\label{eqDWF1}
\end{align}
Here, $c_d $ and $ c_d'$ are the constants depending only on $d$ from Theorem \ref{theoFF} below, $\|\cdot \|_{F1}$ is as in Definition \ref{defnorm}, $\|\cdot\|_1$ is the Schatten 1-norm, and $ \mathcal{C}_d(\R^n) $ is as defined in \eqref{eq:Cd-definition}. 
\end{theorem}

\subsection{Key ingredients}\label{secub}

Let us denote $ s = x_1^2  + \ldots + x_n^2 $ and 
\begin{align*}
	\MQ_{2r}=\{\sigma+u(1-s): \sigma \text{ sos},\ \deg(\sigma)\le 2r,\ u\in \R[x]_{\le 2r-2}\},
\end{align*}
known as the quadratic module of the sphere, truncated at degree $2r$. 

The above improved real QdF theorem for maximally symmetric matrices (Theorem \ref{theoDWr2}) can be obtained as a consequence of the following result of Fang and Fawzi, more precisely,  from their proof technique. 
\begin{theorem}[Fang and Fawzi \cite{Fang_Fawzi_2020}]
	\label{theoFF}
	Let $p$ be an $n$-variate {homogeneous} polynomial of degree $2d$ with $d\ge 1, n\ge 3$. Then, there exist constants $ c_d $ and $ c_d' $ only depending on $ d $, such that, for all $r\ge c_d' n $, it holds that
	\begin{align}
		& p-p_{\min}+ c_d (p_{\max}-p_{\min}) \Big({n\over r}\Big )^2\in \MQ_{2r},\label{eqFF}\\
		& p_{\min}-\sos_r(p) \le c_d  (p_{\max}-p_{\min}) \Big({n\over r}\Big )^2.\label{eqFF1}
	\end{align}
\end{theorem}

So, the above result shows that any homogeneous polynomial that is nonnegative on the sphere can be well approximated by a higher degree sum of squares and it gives a quantitative analysis for the quality of the approximation.  As recalled earlier  in \Cref{sec:prelims} 
(\Cref{propsosP}), the cone of nonnegative polynomials on the sphere is in conic duality with the cone of linear functionals having a representing measure on the sphere. 
A natural question is whether there is an analog on the dual side of the quantitative result of Fang and Fawzi \cite{Fang_Fawzi_2020}. So, one may ask whether a linear functional that is nonnegative on  sums of squares on the sphere can be approximated in ``lower degree" by one that has a representing measure on the sphere. Informally, such a linear functional is given by its positive semidefinite moment matrix and lowering the degree will correspond to taking the partial trace. 

As we explain below, one can derive such a result.
A crucial  reason why such a derivation is possible is that Fang and Fawzi \cite{Fang_Fawzi_2020} use a \emph{linear} operator $ \bfK_r $ to approximate a nonnegative polynomial function on the sphere by a sum of squares. 

Fang and Fawzi use the harmonic expansion \eqref{eqdecharmonic-homogeneous-without-s}, which was introduced in \Cref{secharmonic}. They define an operator $ \bigoplus_{k = 0}^{2r} H_{2k} \to \bigoplus_{k = 0}^{2r} H_{2k} $, which is invertible, $ O(n) $-equivariant and acts on sums of harmonic polynomials of even degree. 
For technical reasons, we prefer to work with the formulation \eqref{eqdecharmonic-homogeneous-with-s} of the harmonic expansion. This way, we stay within the framework of homogeneous polynomials and we  obtain an operator $ \bfK_r\colon R_{2r}\to R_{2r} $, simply by replacing $ \bigoplus_{k = 0}^{2r} H_{2k} $ with the isomorphic space $ R_{2r} $, as   in \eqref{eqdecharmonic}. The operator $ \bfK_r\colon R_{2r} \to R_{2r} $ has the following properties:
\begin{itemize}
	\item[(P1)] $\bfK_r s^r = s^r $. 
	\item[(P2)] If $ f \in \mathcal{P}_r $, then $\bfK_r f \in \Sigma_{r}$. 
	\item[(P3)] $\|\bfK_r^{-1} f - f\|_\infty \le \epsilon_r \|f\|_\infty$ for any $f\in s^{r-d}R_{2d}$. 
\end{itemize}

Ideally, one aims for an operator $ \bfK_r $, where $ \varepsilon_r $ is as small as possible. 
Fang and Fawzi \cite{Fang_Fawzi_2020}  show that for every $ n, d $ and $ r $, there exists an operator $ \bfK_r $ with $ \varepsilon_r \le c_d ({n\over r})^2 $, if $ r\ge c_d'n $. Here, $ c_d $ and $ c_d' $ are some constants, which only depend on $ d $. The following lemma shows that \Cref{theoFF} follows immediately from the existence of a ``good'' linear operator. 

\begin{lemma}\label{lemkernelop1}
	Assume that $\bfK_r$ satisfies (P1), (P2) and (P3). For $p\in R_{2d}$, we have 
	$$ p - p_{\min} + \epsilon_r(p_{\max}-p_{\min})\in \MQ_{2r}.$$
\end{lemma}
\begin{proof}
	Define the polynomial $q=(s^{r-d}p-p_{\min}s^{r})/(p_{\max}-p_{\min})\in s^{r-d}R_{2d} \subseteq R_{2r}$, which satisfies $q_{\min}=0$, $q_{\max}=1$, and $\|q\|_\infty=1$. 
	Set $q_{\epsilon}:=q+\epsilon_r s^{r}\in s^{r-d}R_{2d}$. It suffices to show that $q_\epsilon\in \Sigma_r$, since this implies that $q_\epsilon\in \MQ_{2r}$ (by \cite{deKlerk_Laurent_Parrilo_2005}), which in turn implies the desired relation in the lemma.
	We have $\|\bfK_r^{-1}(q_\epsilon)-q_\epsilon\|_\infty=\|\bfK_r^{-1}(q)-q \|_\infty\le \epsilon_r$, using (P1) for the equality and both (P3) and $\|q\|_\infty=1$ for the inequality. Therefore,  $\bfK_r^{-1}(q_\epsilon )\ge q_\epsilon -\epsilon_r s^{r} =q\ge 0$ on the unit sphere $\oS^{n-1}$. Thus, $\bfK_r^{-1}(q_\epsilon ) \in \MP_r$, since $\bfK_r^{-1}(q_\epsilon )$ is homogeneous. Using (P2),  this implies $q_\epsilon   =\bfK_r ( \bfK_r^{-1}(q_\epsilon  ))\in \Sigma_r$.
\end{proof}

\subsection{The adjoint polynomial kernel operator}
Let $\bfK_r:R_{2r}\to R_{2r}$ be the kernel operator from the previous section.
We now consider the adjoint operator $ \bfK_r^{\ast}\colon R_{2r}^* \to R_{2r}^*$. It maps $L\in R_{2r}^*$ to $\bfK_r^*L \in R_{2r}^*$, defined by
$$
(\bfK_r^*L)(f)= L(\bfK_r f)
\ \text{ for }  
\ f\in R_{2r}.$$
For $L\in R^*_{2r}$ and $d\le r$, we let $L_{|2d}\in R_{2d}^*$ denote the ``restriction'' to $R_{2d}$, defined by $ L_{|2d}(f) = L(s^{r-d}f) $ for any $ f\in R_{2d} $. Before we prove \Cref{theoDWr2}, we need the following   lemmas. 

\begin{lemma}\label{lem:partial-trace-is-restriction}
	Let $ L\in R_{2r}^{*} $ and $ M = L(x^{\ot r}(x^{\ot r})^T)  $ its (tensor-indexed) moment matrix. For $ d\le r $, the moment matrix of $ L_{|2d} $ equals the partial trace $ \tr_{r-d} M $. 
\end{lemma}
\begin{proof}
\tcolblue{	Indeed, 
	$ \tr_{r-d} M = \tr_{r-d} (L(x^{\ot r}(x^{\ot r})^T))= L(\tr_{r-d} (x^{\ot r}(x^{\ot r})^T)) $ (by linearity of $L$). Combining with 
$	\tr_{r-d} (x^{\ot r}(x^{\ot r})^T)) =s^{r-d} x^{\ot d}(x^{\ot d})^T$, we obtain  that 	 
	$ \tr_{r-d} M= L(s^{r-d} x^{\ot d}(x^{\ot d})^T) = L_{|2d}(x^{\ot d}(x^{\ot d})^T)$, where the last identity follows by the definition of $L_{|2d}$.}
	\end{proof}

\begin{lemma}\label{lemkernelopdual}
	The adjoint operator $\bfK_r^*$ satisfies the following properties: for any $L\in R_{2r}^*$,
	\begin{itemize}
		\item[(P1*)] $(\bfK_r^*L)(s^r)=L(s^r)$.
		\item[(P2*)] If $L\ge 0$ on $\Sigma_{r}$, then $\bfK_r^*L$ has a representing measure on $\oS^{n-1}$.
		\item[(P3*)] $\|((\bfK_r^*)^{-1}L-L)_{|2d}\|_1 \le \varepsilon_r \|L\|_1$. 
	\end{itemize}
\end{lemma}

\begin{proof}
	(P1*) follows from (P1). We show (P2*). For this, assume $L\ge 0$ on $\Sigma_{r}$. For $f\in \mathcal{P}_r$, we have
	$\bfK_rf\in \Sigma_r$ by (P2), and thus $L(\bfK_rf)\ge 0$, which gives $(\bfK_r^*L)(f)\ge 0$. Hence, $\bfK_r^*L \ge 0$ on $\mathcal{P}_r$.  By Proposition \ref{propsosP}, we conclude that $\bfK_r^* L$ has a representing measure on $\oS^{n-1}$. We now show (P3*). 
	Let  $f\in s^{r-d}R_{2d}$. Then, 
	\begin{align*}
		|((\bfK_r^*)^{-1} L-L)(f)| = |L((\bfK_r)^{-1}f-f)| \le \|L\|_1 \cdot \|(\bfK_r)^{-1}f-f\|_\infty,
	\end{align*}
	using duality of the norms for the inequality. Combining with (P3) gives (P3*). 
\end{proof}

We now prove Theorem \ref{theoDWr2}.

\begin{proof}[Proof of Theorem \ref{theoDWr2}]
	Assume $M$ is a matrix indexed by $[n]^r$ that is maximally symmetric and positive semidefinite, with $\Tr(M)=1$. Let $ L\in R_{2r}^*$ be the associated linear functional, defined via $ L(x^{\ot r}(x^{\ot r})^T)=M$ (recall 
	Lemma~\ref{lem:moment-matrix-equiv-maxsym}).  Then, $ L(s^r) = \tr M = 1 $ and, by Proposition \ref{propsosP},  $ L\ge 0$ on $\Sigma_{r}$. 
	
	Using properties (P1*) and (P2*), we can conclude that the linear functional $\bfK_r^*L$ has a representing probability measure on $\oS^{n-1}$.  This implies  $\|\bfK_r^*L\|_1\le 1$. Indeed, if $f\in R_{2r}$ satisfies $\|f\|_\infty\le 1$, i.e., $-1\le f\le 1$ on $\oS^{n-1}$, then $ |\bfK_r^*L(f)|\le 1$ and thus $\|\bfK^*_rL\|_1\le 1$.
	Moreover, applying property (P3*) to $\widehat L:= \bfK_r^*L$ yields
	\begin{align}\label{eqboundL}
		\|(L-\bfK_r^* L)_{|2d}\|_1 = \|((\bfK_r^*)^{-1} \widehat L -\widehat L)_{|2d}\|_1 
		\tcolred{\le \varepsilon_r  \|\widehat L_{| 2d}\|_1}\le \varepsilon_r. 
	\end{align}
	Recall that if we choose $ \bfK_r $ to be the operator of Fang and Fawzi, we have $ \varepsilon_r \le c_d ({n\over r})^2 $ for all $ r\ge c_d'n $. Define the following matrix
	$$\tau= (\bfK_r^*L)_{|2d}(x^{\ot d}(x^{\ot d})^T).$$
	Then, $\tau\in \mathcal C_d(\R^n)$, since $\bfK_r^*L$ has a representing probability measure on $\oS^{n-1}$.
	Note also that  $\Tr_{r-d}(M)$ is the tensor-indexed moment matrix of $ L_{|2d} $, by \Cref{lem:partial-trace-is-restriction}.
	Hence, we have
	$$\|\Tr_{r-d}(M)-\tau\|_{F1}=\|(L-\bfK_r^*L)_{|2d}\|_1 \le \varepsilon_r \le c_d ({n\over r})^2. 
	$$
	using the definition of the norm $\|\cdot\|_{F1}$ in Definition \ref{defnorm} for the equality, and (\ref{eqboundL}) for the inequality. This shows the bound for the F1-norm in \eqref{eqDWF1}. 
	Combining with relation (\ref{eqnormLM}) in Lemma \ref{lemnorm} gives the first inequality on the 1-norm in \eqref{eqDWF1}, which completes the proof.
\end{proof}

\section{A banded real QdF theorem for doubly symmetric matrices}\label{secbandedQdF}

A natural question is whether the quantum de Finetti theorem (\Cref{theoQdF}) extends to {\em real} doubly symmetric matrices. 
In view of the description of $ \mathcal{P}_{d}^{\ast} $ from \Cref{propsosP},
one would like to approximate the partial trace of a real doubly symmetric matrix by an element of the set $\mathcal{C}_d(\R^n)$  from (\ref{eq:Cd-definition}).
Unfortunately, this is in general not possible, as we show in Section \ref{sec-counterexample}. However, in Section \ref{sec:banded}, we present a variation of the QdF theorem that does hold for real doubly symmetric matrices. 

\subsection{A counterexample for a real analog of the QdF theorem} \label{sec-counterexample}

Here, we show that the quantum de Finetti theorem (\Cref{theoQdF}) does not extend to real doubly symmetric matrices. 
For this, we present a counterexample (for the case $n=d=2$), which is adapted from the example in \cite[p. 21, Eq. (5.2) and the text thereafter]{Caves_Fuchs_Schack_2002} (for infinite QdF representations) to our setting.

\medskip
For $r\ge 2$, define the matrices
\begin{align}\label{eqrhor}
	\rho_r= {1\over 2} (uu^*)^{\ot r}+{1\over 2} (\overline u \overline u^*)^{\ot r}, \ \text{ where } u={1\over \sqrt 2}\left(\begin{matrix} 1\cr \bfi \end{matrix}\right)\in \C^2.
\end{align}
Then, $\rho_r$ is $r$-doubly symmetric, real valued and positive semidefinite, with trace 1. However, as we now show, there cannot exist elements $\tau_r\in \mathcal C_2(\R^n)$ that are asymptotically `close' to the partial traces $\Tr_{r-2}(\rho_r)$.

\begin{proposition}\label{prop:nonexistence-of-real-qdf}
	For any $r\ge 2$, consider the matrices $\rho_r$ from (\ref{eqrhor}), that are real, $r$-doubly symmetric, positive semidefinite, with trace 1. There do not exist elements $\tau_r\in \mathcal C_2(\R^n)$ ($r\ge 2$) for which $\lim_{r\to\infty} \|\Tr_{r-2}(\rho_r)-\tau_r\|_1=0$.
\end{proposition}

\begin{proof}
	First, observe that $\Tr_{r-2}(\rho_r)=\rho_2$ (since $u$ is a unit vector).  Assume for contradiction that such $\tau_r\in \mathcal C_2(\R^n)$ would exist. As the set $\mathcal C_2(\R^n)$ is compact, the sequence $(\tau_r)_{r\ge 2}$ admits a converging subsequence, converging to $\tau\in \mathcal C_2(\R^n)$.  
	At the limit, we obtain that $\rho_2=\tau\in \mathcal C_2(\R^n)$. This yields a contradiction, since any element of $\mathcal C_2(\R^n)$ is maximally symmetric, while $\rho_2$ is not maximally symmetric \tcolred{since it is not real-valued}. 
\end{proof}

\subsection{A banded real QdF theorem}\label{sec:banded}

So, no analog of the QdF theorem exists, for approximating partial traces of real doubly symmetric matrices by elements of $\mathcal C_d(\R^n)$. 
On the other hand, the assumption of \Cref{theoDWr2}   requiring  a maximally symmetric matrix is very strong.  We now state an approximation result  that applies to real doubly symmetric matrices,  but reaches a weaker conclusion, namely approximation \tcolblue{of the maximal symmetrization of the partial trace of $M$} by a {\em scaled }element of $\mathcal C_d(\R^n)$  and a weaker approximation rate. So, the assumption of \Cref{thm:banded-qdf} is weaker
than in \Cref{theoDWr2}, but the approximation guarantee is also weaker. Hence, the two results are incomparable.

\begin{theorem}[Banded real QdF for real doubly symmetric matrices]\label{thm:banded-qdf} 
	Let $M$ be a real symmetric matrix indexed by $[n]^r$. Assume $M$ is doubly symmetric, positive semidefinite 
	and $\Tr(M)=1$. Then, for any integer $1\le d\le r$, there exist a scalar $\alpha \in [\alpha_d, 1]$ and an element $\tau\in \mathcal C_d(\R^n)$ such that 
	\begin{align}\label{eqbanded}
		\|\Pi_{\maxsym}(\tr_{r-d} M) - \alpha \tau\|_{1} \le \sqrt{N_{n,d}}{4d(n-1)\over r+1}.
	\end{align}
	Here, $ N_{n, d} = \binom{n+d-1}{d} $ is the dimension of $ S^d(\R^n) $ and $\alpha_d={2^d/{2d\choose d}}$.
\end{theorem}

The following result is a crucial ingredient in the proof of 
\Cref{thm:banded-qdf}; \ATB{we extract it from the proof of \cite[Theorem 4.1]{Johnston_Lovitz_Vijayaraghavan_2023} and adapt it to our needs.}
\tcolblue{For clarity and completeness, we give a sketch of proof in Appendix \ref{appendix-proof-lemma}.}

\begin{lemma}[adapted from \cite{Johnston_Lovitz_Vijayaraghavan_2023}]\label{lemJL}
	Assume $\tau\in \mathcal C_d(\C^n)$ with $\Tr(\tau )=1$, and define 
	$\alpha_d={2^d/ {2d\choose d}}$, 	$\alpha= \Tr(\PiMS(\tau))$. Then,  ${1\over \alpha}\PiMS(\tau) \in \mathcal C_d(\R^n)$, and
	$\alpha_d \le \alpha\le 1.$
\end{lemma}

Lovitz and Johnson \cite{Johnston_Lovitz_Vijayaraghavan_2023} use this result (within the proof of their Theorem 4.1, \MoL{recalled later as Theorem \ref{theoJLV}}) to reduce the analysis 
of  the real polynomial optimization problem to its complex analog, for which there is a well-established spectral hierarchy based on the QdF theorem. We will use these same ingredients in a different way in order to show the above banded real QdF theorem (\Cref{thm:banded-qdf}), which is then used   in \Cref{secJLV}  to recover the analysis of the spectral bounds, \tcolblue{however with different more explicit constants and better dependency on the polynomial.}

\begin{proof}[Proof of Theorem \ref{thm:banded-qdf}]
	First, we use the quantum de Finetti theorem, in the form of \Cref{corQdF}, to pick $ \tau \in \mathcal{C}_d(\C^n) $ such that  
	\begin{align}\label{eqtau0}
		\|\tau - \tr_{r-d} M \|_{1} \le \frac{4d(n-1)}{r+1}. 
	\end{align}
	We claim: 
	\begin{align}\label{eqtau2}
		\|\PiMS(\tr_{r-d} M) - \PiMS(\tau)\|_{1} \le  \sqrt{N_{n,d}} \|\Tr_{r-d}(M) - \tau\|_{1}.
	\end{align}
	Indeed, for  the Hermitian matrix $A=\tr_{r-d} M -\tau$, we have
	$$ \|\PiMS(A)\|_1 \le \sqrt{N_{n,d}} \  \|\PiMS(A)\|_2 \le \sqrt{N_{n,d}} \  \|A\|_2 \le  \sqrt{N_{n,d}}\ \|A\|_1.$$
	The first inequality follows from relation (\ref{eqnorm12}) applied to the matrix $\PiMS(A)$, which has at most $N_{n,d}$ nonzero eigenvalues since it is maximally symmetric. The second inequality follows from the fact that the Schatten 2-norm coincides with the Frobenius norm and thus it can only decrease under projection on a subspace. The last inequality follows using relation (\ref{eqnorm12}) for $A$.

	Define the matrix $\tau':=\PiMS(\tau) $ and set $\alpha:=\Tr(\tau')$. Then, by Lemma \ref{lemJL}, $\tau'':={1\over \alpha} \tau'\in \mathcal C_d(\R^n)$ and $\alpha_d\le \alpha\le 1$. Combining (\ref{eqtau0}) and  (\ref{eqtau2}), we obtain
	\begin{align*}
		\|\PiMS(\tr_{r-d} M) -\alpha \tau'' \|_{1} \le \sqrt{N_{n,d}} {4d(n-1)\over r+1},
	\end{align*}
	which completes the proof.
\end{proof}

\begin{remark}\label{rem:bandsize-optimality}
	Note that in order to reach a good approximation of the matrix $ A:= \PiMS (\tr_{r-d} M) $ by $(\alpha,\tau)$  in \Cref{thm:banded-qdf}, the scalar $ \alpha $ must be very close to the trace of  $ A$. To see this, use the following inequality: 
	\begin{align*}
		\|A-\alpha\tau\|_{1} = \max_{\|B\|_{\infty} = 1} \langle A-\alpha\tau, B \rangle \ge \langle A-\alpha\tau, \pm I_n^{\otimes d} \rangle = |\Tr A-\alpha|.
	\end{align*}
	Hence, if $ \alpha$ and $ \tau $ satisfy relation (\ref{eqbanded}), then
	\begin{align*}
		|\Tr A - \alpha| \le \sqrt{N_{n,d}}{4d(n-1)\over r+1} \le \epsilon,
	\end{align*}
	where the last inequality holds for any $\epsilon>0$ and any $r$ large enough.
	By \Cref{thm:banded-qdf} there exists a scalar $ \alpha' $ satisfying $\alpha_d\le \alpha'\le 1$ and $|\Tr A-\alpha'|\le \epsilon$. This implies $\alpha_d - \varepsilon < \Tr A < 1 + \varepsilon $. In turn, it now follows that any scalar $ \alpha $ with $ \|A-\alpha\tau\|_{1} \le \sqrt{N_{n,d}}{4d(n-1)\over r+1}  $ must satisfy $ \alpha_d-2\varepsilon <  \alpha <  1+2\varepsilon $. This shows that  it is not possible to improve the approximation of $A$ in \Cref{thm:banded-qdf} by making the ``band'' larger.  
\end{remark}

\section{The spectral hierarchy of \johnstonlovitz{}}\label{secJLV}

\MoL{\johnstonlovitz{} \cite{Johnston_Lovitz_Vijayaraghavan_2023} proposed recently a hierarchy of spectral bounds $\spec_r(p)$ for the minimum value $p_{\min}$ of a homogeneous polynomial $p$ over the unit sphere,\footnote{ Their hierarchy  applies more generally to optimization of  a multivariate tensor over a spherical Segre-Veronese variety. We only focus here on their hierarchy for polynomials on the sphere.} 
see Section \ref{secJLV-intro} below for their definition and relationship to the  moment-sos bounds.  Throughout, let $p\in R_{2d}$, $s=x_1^2+\cdots +x_n^2$, and let 
\begin{align}\label{eqQpM}
	Q(p) = \maxsym(p), \quad  M = \maxsym(s^d) 
\end{align}
denote  the maximally symmetric  matrices representing the polynomials $ p $ and $s^d$, respectively (as in Definition \ref{defGramtensor}). They show the following convergence rate.
}

\MoL{\begin{theorem}\cite[Theorem 4.3]{Johnston_Lovitz_Vijayaraghavan_2023}\label{theoJLV}
Let $p\in R_{2d}$,  let $Q(p)$ (resp., $M$) be the maximally symmetric matrix representing $p$
(resp., $s^d$), and $\alpha_d=2^d/{2d\choose d}$. For $r\ge d$, 
\begin{align*}
p_{\min}-\spec_r(p) \le \|Q(p)\|_\infty \big(1+{\lambda_{\max}(M)\over \lambda_{\min}(M)}\big) {1\over \alpha_d} {4d(n-1)\over r+1},
\end{align*}
\end{theorem}}

\MoL{Several questions arise that we will address in this section.}

\noindent
\MoL{$\bullet$ Is it possible to show an improved rate, better than $O(1/r)$? We cannot answer this, but what we can do is show that one cannot hope for a better rate than $O(1/r^2)$. For this, we construct an instance whose rate is in $\Omega(1/r^2)$ (see Theorem \ref{thm:lower-bound-spectral}). 
}

\smallskip\noindent
\MoL{$\bullet$ 
The constant in the error estimate of Theorem \ref{theoJLV} involves the condition number $\kappa(M)={\lambda_{\max}(M)\over \lambda_{\min}(M)}$  whose exact value is not known (see \cite{Johnston_Lovitz_Vijayaraghavan_2023} and Remark \ref{remcompare} below).
Is it possible to have another constant with explicit dependency  in terms of $n,d$? This is what we offer in Theorem \ref{thm:analysis-spectral-hierarchy}, where we replace $\kappa(M)$ by $N_{n,d}={n+d-1\choose d}$. For this, we use the banded real QdF (Theorem  \ref{thm:banded-qdf}) instead of the QdF (Theorem \ref{theoQdF}) used by Lovitz and Johnston \cite{Johnston_Lovitz_Vijayaraghavan_2023} for showing Theorem~\ref{theoJLV}. 
}

\smallskip\noindent
\MoL{$\bullet$ Can one show an error estimate in terms of the range of values $p_{\max}-p_{\min}$ taken by  $p$ over the unit sphere? Such estimate would be in line with the analysis of the moment-sos bounds (in Theorem \ref{theoFF}) and fit with what is customary in the optimization literature. We offer such an analysis in Theorem \ref{thm:pmax-pmin-spectral}. 
}

\smallskip\noindent
\MoL{$\bullet$ Does the spectral hierarchy have generic finite convergence, as is the case for   the moment-sos bounds? We give a negative answer (for degree 4) in Theorem \ref{thm:no-finite-convergence-n}.
}

\subsection{The spectral bounds}\label{secJLV-intro}

We introduce the spectral hierarchy of \johnstonlovitz{} \cite{Johnston_Lovitz_Vijayaraghavan_2023}. 
Let $p\in R_{2d}$ and recall the matrices $Q(p)$   and $M$ from (\ref{eqQpM}). For an integer $ r\ge d $, define the matrices 
\begin{align}\label{eqQrpM}
	Q_r(p) = \Pi_{r}(Q(p) \otimes I_n^{\otimes (r-d)})\Pi_{r}, \\
	M_r= \Pi_{r}(M \otimes I_n^{\otimes (r-d)})\Pi_{r}. 
\end{align}
So,  $ Q_r(p) $ and $M_r$ are   $ r $-doubly symmetric matrices    representing  $ s^{r-d}p $ and $s^r$, respectively. That is, $\langle Q_r(p), x^{\ot r}(x^{\ot r})^T\rangle = s^{r-d}p$ and $
\langle M_r, x^{\ot r}(x^{\ot r})^T\rangle= s^r$ hold.
Then, the  spectral bound  $\spec_r(p)$ at the $ r $-th level is defined as the smallest generalized eigenvalue of the matrix $ Q_r(p) $ with respect to $ M_r $. We gather some equivalent formulations for $\spec_r(p)$. 

\begin{definition}
	The spectral bound $\spec_r(p)$  is defined by any of the following equivalent programs:
	\begin{align}
		\spec_r(p)& =\max_{\nu\in\R} \{ \nu: Q_r(p)-\nu M_r\succeq 0\},\label{eqnu1} \\
		& = \min_{v\in S^r(\R^n) } \{v^T Q_r(p)v: v^T M_rv=1\}, \label{eqnu2}\\
		& = \min_{X } \{\langle Q_r(p), X\rangle: X \  r\text{-doubly symmetric},\ X\succeq 0,\ \langle M_r,X\rangle =1\}.\label{eqnu3}
	\end{align}
\end{definition}

Lovitz and Johnston  \cite{Johnston_Lovitz_Vijayaraghavan_2023} show that the parameters $\spec_r(p)$ provide lower bounds for $p_{\min}$ that are at most as strong as the sos parameters $\sos_r(p)$ from (\ref{eqsosr}). 

\begin{proposition}\label{prop:jlv-relaxes-pmin} \cite{Johnston_Lovitz_Vijayaraghavan_2023}
We have $ \spec_r(p)\le \spec_{r+1}(p)$ and $\spec_r(p)\le \sos_r(p)\le p_{\min}$ for $r\ge d$ \tcolblue{and $p\in R_{2d}$.}
\end{proposition}

\tcolblue{For completeness and clarity, we give the proof in Appendix \ref{appendix-proof-proposition}.}
\tcolblue{To ease comparison between the spectral and sos bounds,} we   recall some analogous equivalent formulations of  the parameters $ \sos_{r}(p) $.

\begin{lemma}\label{lemsosmoment}
	The sos bound $\sos_r(p)$ can be equivalently defined by any of the following programs:
	\begin{align}
		&\sos_{r}(p)   = \max\{\lambda: \lambda\in \R,\ s^{r-d}p-\lambda s^{r} \text{ is sos}\}, \label{eqsosr-alternative}\\
		& = \min\{L(ps^{r-d}): L\in R_{2r}^*,\ L(s^r)=1, \ L\ge 0 \text{ on } \Sigma_r\},\label{eqsosr-L}\\
		&= \min \{\langle Q(p)\otimes I_n^{(r-d)}, A\rangle: A\in \MaxSym_r(\R^n),  A\succeq 0,  \Tr(A)=1\}.\label{eqsosr-M}
	\end{align}
\end{lemma}

\begin{proof}
For the equivalence between (\ref{eqsosr}) and (\ref{eqsosr-alternative}), see, e.g.,  \cite{deKlerk_Laurent_Parrilo_2005}.
The equivalence of (\ref{eqsosr-alternative}) and (\ref{eqsosr-L}) follows from conic duality, and use Proposition \ref{propsosP} for the equivalence between (\ref{eqsosr-L}) and  (\ref{eqsosr-M}).
\end{proof}

\subsection{Convergence analysis of the spectral bounds}\label{sec:convergence-analysis-spectral}
\raggedbottom

\MoL{Lovitz and Johnston \cite{Johnston_Lovitz_Vijayaraghavan_2023} show their convergence analysis in $O(1/r)$ 
(recall Theorem \ref{theoJLV}) using the QdF theorem (Theorem \ref{theoQdF}). In this section, we use our new  banded real QdF theorem (Theorem \ref{thm:banded-qdf}) to give an alternative proof of this fact. First, in Theorem \ref{thm:analysis-spectral-hierarchy}, we replace the (unknown) constant $ \lambda_{\max}(M)/\lambda_{\min}(M)$ appearing in the $O(\cdot)$ notation, by the explicit constant $N_{n,d}$. The dependence on $p$ in both results is via the eigenvalues of the representation matrix $Q(p)$ (via $\|Q(p)\|_\infty$ in Theorem \ref{theoJLV} and via the parameter $\gamma(Q(p))$ from (\ref{eqgamma}) in Theorem \ref{thm:analysis-spectral-hierarchy}). Second, in Theorem \ref{thm:pmax-pmin-spectral}, we replace  this dependence on the spectrum of $Q(p)$  by the value  range $p_{\max}-p_{\min}$ of $p$.}

\begin{theorem}
\label{thm:analysis-spectral-hierarchy}
	Let $p\in R_{2d}$. 	For any $r\ge d$, we have
	\begin{align*}
		p_{\min} - \spec_r(p) \le  \gamma(Q(p)) \cdot {\sqrt{N_{n,d}}\over \alpha_d} \cdot  \frac{4d(n-1) }{r+1}, 
	\end{align*}
	where $ \alpha_d = {2^d/ {2d\choose d}}$ is the constant from \Cref{thm:banded-qdf}, $ N_{n,d} = \binom{n+d-1}{d} $, and 
	\begin{align}\label{eqgamma}
		\gamma(Q(p)) = \max\{\lambda_{\max}(Q(p)) - \spec_d(p), p_{\min} - \lambda_{\min}(Q(p))\}.
	\end{align}
\end{theorem}

\begin{proof} 
	To simplify notation, set $Q:=Q(p)$.
	Let $ v\in S^r(\R^n) $ be an optimal solution of the program (\ref{eqnu2}).	
	The  following identity (taken from \cite{Johnston_Lovitz_Vijayaraghavan_2023}) holds, useful for later use:
	\begin{align}\label{eq:eigenvector-calculation}
		0 &= \langle Q_r(p) - \spec_{r} M_r, vv^{T}  \rangle =  \langle (Q - \spec_r M) \otimes I_n^{\otimes r-d},  vv^{T}  \rangle \\
		&= \langle Q - \spec_{r} M, \tr_{r-d} (vv^{T})  \rangle \nonumber\\
		&= \langle \Pi_{\maxsym}(Q - \spec_{r} M), \tr_{r-d} vv^{T}  \rangle \nonumber\\
		&= \langle Q - \spec_{r} M, \Pi_{\maxsym}(\tr_{r-d} vv^{T}) \rangle. \nonumber
	\end{align}
	Here, we use the fact that $vv^T$ is doubly $r$-symmetric in the first line (to eliminate the projections $\Pi_r$ in $Q_r(p) - \spec_{r} M_r$), the definition of the partial trace in the second line and, in the  third line,  the fact that $ Q-\spec_{r}M $ is maximally symmetric. 
	
	By \Cref{thm:banded-qdf}, there exist $ \tau\in \mathcal{C}_d(\R^n) $ and a scalar $ \alpha \in [\alpha_d,2] $ such that 
	\begin{align}\label{eqeval1}
		\| \tau-\alpha^{-1}\Pi_{\maxsym}(\tr_{r-d} vv^{T})\|_{1} \le {{\sqrt{N_{n,d}}} \over \alpha_d}{4d(n-1)\over r+1}.
	\end{align}

	By definition, $p_{\min} =\min_{\rho\in \mathcal C_d(\R^n)}\langle Q(p),\rho\rangle \le \langle Q(p), \tau\rangle.$
	We proceed to evaluate the range $p_{\min}-\spec_r(p)$, setting $\spec_r=\spec_r(p)$ for simpler notation:
	\begin{align}\label{eq:jlv-estimate-chain}
	\begin{split}
		p_{\min} - \spec_r 
		&\leq \langle Q, \tau \rangle - \spec_r\\
		&= \langle Q - \spec_r M, \tau \rangle \\ 
		&= \langle Q - \spec_r M, \tau - \alpha^{-1}\Pi_{\maxsym}(\tr_{r-d} vv^{T}) \rangle \\ 
		&= \langle Q - \spec_r I_n^{\otimes d}, \tau - \alpha^{-1}\Pi_{\maxsym}(\tr_{r-d} vv^{T}) \rangle \\ 
		&\leq \|Q - \spec_r I_n^{\otimes d}\|_{\infty} \cdot \|\tau - \alpha^{-1}\Pi_{\maxsym}(\tr_{r-d} vv^{T})\|_{1}  
	\end{split}
	\end{align}
	The first equality goes as follows:
	$\langle M,\tau\rangle=\langle \MaxSym(s^d),\tau\rangle =\langle I_n^{\ot d},\tau\rangle =1$, since $\tau $ is maximally symmetric with trace 1. For the second equality, use (\ref{eq:eigenvector-calculation}) to see that we added a zero term. For the third equality, use again that $\tau$ is maximally symmetric, so  we can replace $Q-\spec_r M$ by $Q-\spec_r I_n^{\ot d}$. Finally, the last inequality uses H\"older's inequality (for the Schatten 1-norm and $\infty$-norm).

	In view of (\ref{eqeval1}), there remains only to show 
	\begin{align}\label{eqQinfty}
		\|Q - \spec_r I_n^{\otimes d}\|_{\infty}\le \gamma(Q).
	\end{align}
	By definition, $\|Q-\spec_r I_n^{\ot d}\|_\infty =  \max\{\lambda_{\max}(Q)-\spec_r, \spec_r-\lambda_{\min}(Q)\}$. Note that
	$$  \spec_d \le \spec_r \le p_{\min} \le p_{\max} \le \lambda_{\max}(Q). $$ 
	\MoL{For the right most inequality, say $p_{\max}=p(x)$ with $x\in\mathbb S^{n-1}$; then,  $p_{\max}=p(x)= \langle Q, (xx^T)^{\ot d}\rangle  \le \lambda_{\max}(Q)$.} Hence, $\lambda_{\max}(Q)-\spec_r\le \lambda_{\max}(Q)-\spec_d$ and $\spec_r-\lambda_{\min}(Q) \le p_{\min}-\lambda_{\min}(Q)$, which shows (\ref{eqQinfty}), as desired.
\end{proof}

\begin{remark}\label{remcompare}
In comparison to the factor $\gamma(Q(p)) \cdot {\sqrt{N_{n,d}}\over \alpha_d}$ appearing in Theorem~\ref{thm:analysis-spectral-hierarchy}, the corresponding error term in Theorem \ref{theoJLV} is $\|Q(p)\|_\infty {1 +\kappa(M)\over \alpha_d}$, where $\kappa(M)={\lambda_{\max}(M)\over \lambda_{\min}(M)}$ and $M=\MaxSym(s^d)$. Compared to $ N_{n,d}^{1/2}= {n+d-1\choose d}^{1/2}$, the constant $\kappa(M)$ is not known explicitly. Lovitz and Johnston \cite{Johnston_Lovitz_Vijayaraghavan_2023} observe that numerics suggest $\kappa(M)={n/2 + d-1 \choose \lfloor d/2\rfloor}$, which, if true, would imply that both $\sqrt{N_{n,d}}$ and $\kappa(M)$ are in the same order $O(n^{d/2})$ for fixed $d$. 
\end{remark}

\MoL{We next revisit the proof of Theorem~\ref{thm:analysis-spectral-hierarchy}  and replace   the dependency on $p$ by the value range $ p_{\max}-p_{\min} $. Such a dependency is indeed preferable compared to the parameters $\|Q(p)\|_{\infty}$ and $\gamma(Q(p))$, which do not behave well under shifting $p$ by a scalar multiple of $s^d$ (as mentioned in Appendix \ref{appendix-bad-gamma}).}

\begin{theorem}\label{thm:pmax-pmin-spectral}
	There exist   constants $ c_{n, d}>0$, only depending on $ n $ and $ d $, \MoL{and $r_0\in \oN$}  such that 
	\begin{align}\label{eqpmaxpmin}
		 p_{\min}- \spec_r(p) \le  (p_{\max} - p_{\min}) \cdot\frac{c_{n, d} }{r} \ \ \MoL{\text{ for all } r\ge r_0.}
	\end{align}
\end{theorem}

\begin{proof}
\MoL{We begin with observing that both $p_{\min}-\spec_r(p)$ and $p_{\max}-p_{\min}$ remain invariant under shifting $p$ by a scalar multiple of $s^d$. Hence, the inequality (\ref{eqpmaxpmin}) holds for $p$ if and only if it holds for $p+a s^d$ for any scalar $a\in\oR$.
So, we may assume 
\begin{align}\label{eqharmp}
p\in \bigoplus_{k=1}^d s^{d-k}H_{2k}.
\end{align}
 Indeed, if $p=cs^d+\sum_{k=1}^d s^{d-k}p_{2k}$ with $p_{2k}\in H_{2k}$ and $c\in \oR$, is the harmonic decomposition of $p$ (as in (\ref{eqdecharmonic-homogeneous-with-s})), then we may replace $p$ by $p-cs^d$. Hence, we have 
 $$p_{\min}<0<p_{\max}.$$ 
 Indeed, in view of (\ref{eqharmp}), 
 $p$ is orthogonal to the constant polynomial in $H_0$. That is, $\int_{\mathbb S^{n-1}} p(x)d\mu(x)=0$, where $\mu$ is the Haar measure on the sphere (recall Remark \ref{remharmonic}).
Since  $\lambda_{\min}(Q(p))\le p_{\min}<0<p_{\max} \le \lambda_{\max}(Q(p))$, this implies 
\begin{align*}
\|Q(p)\|_{\infty} \le \lambda_{\max}(Q(p))-\lambda_{\min}(Q(p)).
\end{align*}
Next, we show that the spectral range of $Q(p)$ can be bounded in terms of the value range of $p$: there exists a constant $c'_{n,d}$ (depending only on $n$ and $d$) such that 
\begin{align}
\label{eqrangeQq}
\lambda_{\max}(Q(q))-\lambda_{\min}(Q(q))\le c'_{n,d} (q_{\max}-q_{\min} )
\end{align}
for all $q\in \bigoplus_{k=1}^d s^{d-k}$. For this, 	consider the function 
	\begin{align*}
		\Phi\colon \bigoplus_{k=1}^d H_{2k}  \to \R,\  q \mapsto \frac{\lambda_{\max}(Q(q)) - \lambda_{\min} (Q(q))}{q_{\max} - q_{\min}}.
	\end{align*}
	Observe that $\Phi$ is well-defined, since its domain does not contain constant functions, and we have $ \Phi\ge 1 $. 
	Since $ \Phi $ is continuous and homogeneous of degree zero, it attains its  maximum value, denoted $ c_{n, d}' $, and the inequality (\ref{eqrangeQq}) follows. We can now conclude the proof.}
	
\MoL{	For this, we use   the last line of  \Cref{eq:jlv-estimate-chain}. So,   it suffices to upper bound each of the two terms 
$\|\tau - \alpha^{-1}\Pi_{\maxsym}(\tr_{r-d} vv^{T})\|_{1}$ and $\|Q-\spec_r I_n^{\ot d}\|_\infty$ (with $Q=Q(p)$).
The first term involving the $1$-norm can be upper bounded  by $ {{\sqrt{N_{n,d}}} \over \alpha_d}\cdot {4d(n-1)\over r+1} $ as in (\ref{eqeval1}). For the second term involving the $\infty$-norm, we use the triangle inequality combined with the fact that $\|I\|_{\infty}=1$ and $|\spec_r|\le |\spec_d|$ (since $\spec_d\le \spec_r\le p_{\min}<0$), to obtain $$\|Q-\spec_r I\|_{\infty} \le \|Q\|_{\infty} +|\spec_d|.$$
Observe that ${M\over \lambda_{\min}(M)} \succeq I$ and $I\succeq -{Q\over \|Q\|_\infty}$ 	(since
$\|Q\|_{\infty}\ge -\lambda_{\min}(Q)$). This implies $Q+ {\|Q\|_{\infty}\over \lambda_{\min}(M)} M\succeq  0$, and thus $|\spec_d|\le {\|Q\|_{\infty}\over \lambda_{\min}(M)}$ by the definition of the parameter $\spec_d=\spec_d(p)$. Putting things together and combining with (\ref{eqrangeQq}), we obtain
$$\|Q-\spec_r I\|_{\infty} \le \|Q\|_{\infty}\big(1+ {1\over \lambda_{\min}(M)}\big) \le c'_{n,d} \big(1+{1\over \lambda_{\min}(M)}\big) (p_{\max}-p_{\min}),$$
giving
$p_{\min}-\spec_r(p) \le c_{n,d}(p_{\max}-p_{\min})$,  with
$c_{n,d}= c'_{n,d} \big(1+{1\over \lambda_{\min}(M)}\big) {\sqrt{N_{n,d}}\over \alpha_d} 4d(n-1).$}
\end{proof}

\subsection{A complexity lower bound for the spectral hierarchy}\label{secChoiLam}

As seen earlier, one can show a convergence rate in $O(1/r)$ for  the spectral hierarchy $\spec_r(p)$.  Whether one can show a sharper analysis in $O(1/r^k)$ for some $k>1$ is not known. In this section we show that the best one can hope for is $k=2$. 
For this, we are going to construct an explicit form $ p $ of degree $ 4 $ in five  variables for which $ \spec_r(p) = \Omega(1/r^2) $. 

\medskip
We start by picking a  quartic form  in four variables that is nonnegative but not a sum of squares.
A concrete example is given by the \emph{Choi-Lam form} $ p_{CL} $, 
\begin{align}\label{eqpCL}
	p_{CL}(x) = x_1^2x_2^2+x_1^2x_3^2+ x_2^2x_3^2 + x_4^4 -4 x_1x_2x_3x_4.
\end{align}
For simpler notation,    denote by  $ \spec_r = \spec_r(p_{CL})$   the spectral  bounds, 
as defined in (\ref{eqnu1}), and   by $ \sos_{r}=\sos_r(p_{CL}) $  the moment-sos lower bounds, as in \eqref{eqsosr}. 
Then, $\spec_r\le \sos_r\le (p_{CL})_{\min }=0$ for any integer $r\ge 2$.

The Choi-Lam form $ p_{CL}$ is famously known to be a nonnegative polynomial, which is not a sum of squares (see \cite{Choi_Lam_1977}). Hence, {\em no} matrix representation of $ p_{CL} $ is positive semidefinite.
This property is stronger than necessary. As it will turn out, our strategy for proving the lower bound $\Omega(1/r^2)$ for  the spectral hierarchy works with any nonnegative quartic form whose  maximally symmetric representation is not positive semidefinite. 

We use one small trick, which is to choose an unnatural representation of   $ p_{CL} $. While $ p_{CL} $ is intrinsically a polynomial in four variables $x=(x_1,x_2,x_3,x_4)$, one can artificially view it as a polynomial $p$ in five variables $ \underline{x} := (x_1,x_2,x_3,x_4,u) $, where $u$ is an additional variable, by setting $p(\underline x)=p_{CL}(x)$.
If, as a 4-variate polynomial,  $ p_{CL} $ is represented by a matrix $ Q^{(4)} \in \R^{[4]^2\times [4]^2} $, then, as a 5-variate polynomial in   $ \underline{x} = (x,u) $, it is represented by the matrix 
\begin{align}\label{eqmatQ}
	\begin{pmatrix}
		Q^{(4)} & 0\\
		0^{T} & 0
	\end{pmatrix} \in  \R^{[5]^2\times [5]^2}.
\end{align}
Note that the same holds   for the maximally symmetric representations of $p_{CL}$  in four variables and of $p $ as a polynomial in five variables: 
\begin{align}\label{eqmatMQ}
	\maxsym(p)  = \begin{pmatrix}
		\maxsym(p_{CL}) & 0\\
		0^{T} & 0
	\end{pmatrix}.
\end{align}
The reason is again that $ p $ is only supported on monomials, which do not depend on the fifth variable $ u $. 

\begin{theorem}\label{thm:lower-bound-spectral}
	Let $\spec_r(p)$ denote the spectral bound of order $r$ for the Choi-Lam form $p \in \R[x_1,x_2,x_3,x_4,u]$. Then, we have	$p_{\min}=0$ and $\spec_r(p)=\Omega(1/r^2)$.
\end{theorem}

\begin{proof}
	Let the variables be ordered as $ \underline{x} = (x,u)$, where $ x = (x_1,\ldots,x_4) $ and  $ u $ is the fifth variable. Let $r\ge 2$. Set $ s = x_1^2  + x_2^2+x_3^2 + x_4^2 + u^2 $ and 
	\begin{align}\label{eqmats}
	\begin{split}
	 & Q :=   \maxsym(p),\quad  M := \maxsym(s^2), \\
	 &  Q_r := \Pi_r (Q\otimes I_5^{\otimes (r-2)}) \Pi_r,\quad  M_r :=  \Pi_r (M\otimes I_5^{\otimes (r-2)}) \Pi_r.
	\end{split}
	\end{align} 
	So, $Q$ has the block-form as in (\ref{eqmatQ}), (\ref{eqmatMQ}), where $Q^{(4)} :=  \MaxSym(p_{CL}) \in \R^{[4]^2 \times [4]^2} $ and the remaining blocks are filled by zeros. As in (\ref{eqPir}),  $\Pi_2\in \R^{[5]^2\times [5]^2}$ denotes the projection onto the symmetric subspace $S^2(\R^5)$.	The space $\R^{[4]^2}$ can be seen as a subspace of $\R^{[5]^2}$ corresponding to the first four variable indices. Let us write $ \Pi_2^{(4)} \in \R^{[4]^2 \times [4]^2} $ for the restriction of $ \Pi_2$ to the subspace $ \R^{[4]^2} $. 

	Recall  $\spec_r=\spec_r(p)$ for short. By the definition, we have $\spec_r\le p_{\min}=0$ and 
	\begin{align*}
		Q_r + |\spec_r| M_r = \Pi_r ((Q+ |\spec_r|M) \otimes I_5^{\otimes (r-2)}) \Pi_r \succeq 0.
	\end{align*}
	By using  $ M\preceq \lambda_{\max}(M) I_5^{\otimes 2} $, we obtain that 
	\begin{align*}
		G_r :=  \Pi_r ((Q+ \lambda_{\max}(M)|\spec_r|I_5^{\otimes 2}) \otimes I_5^{\otimes (r-2)}) \Pi_r \succeq 0.
	\end{align*}
	In particular, every principal submatrix of $ G_r $ is positive semidefinite. We will now examine \tcolblue{in detail} the following principal submatrix 
	\begin{align}\label{eqmatH}
		H:= ((G_r)_{(i_1,i_2,5,\ldots,5), (j_1,j_2,5,\ldots,5)})_{i_1,i_2,j_1,j_2 \in \{1,\ldots,4\}} \in \R^{[4]^2\times [4]^2}
	\end{align}
	of $ G_r $, where the first two row indices $ i_1, i_2 $ and the first two column indices $ j_1,j_2 $ each range from $ 1 $ to $ 4 $, while the other $ r-2 $ row and column indices are set equal to $ 5 $.

	Recall the notation $ \underline{i}^{\sigma} := (i_{\sigma(1)},\ldots,i_{\sigma(r)}) $ for any $\sigma\in \MFS_r$. For $\ui,\uj\in [5]^r$, using (\ref{eqPir}), we may expand 
	\begin{align*}
		(G_r)_{\underline{i},\underline{j}} = \left(\frac{1}{r!}\right)^2 \sum_{\sigma \in \mathfrak{S}_{r}}\sum_{\tau \in \mathfrak{S}_{r}}  (Q \otimes I_5^{\otimes (r-2)} + \lambda_{\max}(M)|\spec_r|I_5^{\otimes r})_{\underline{i}^{\sigma}, \underline{j}^{\tau}}.
	\end{align*}
	In order to compute the submatrix $ H $, we set the last $ r-2 $ indices of both $ \underline{i} $ and $ \underline{j} $ equal to $ 5 $ and the other two indices not equal to $ 5 $. Hence, on the right hand side, we have to count how many permutations ``survive'' when we do this, in the sense of giving a nonzero contribution to $ H $. 
	We may write 
	\begin{align}\label{eqmatH2}
	 H = A + \lambda_{\max}(M)|\spec_r|B,
	 \end{align}
	  where $ A $ is the contribution arising from $ Q \otimes I_5^{\otimes (r-2)} $ and $ B $ is the contribution from $ I_5^{\otimes r}  $. Since $ I_5^{\otimes r}  $ is diagonal, the only nonzero entries of $ B $ are at indices of the form $ (i_1i_2,i_1i_2) $ and $ (i_1i_2,i_2i_1) $. 
	For later use, notice that $ (\Pi_{2})_{i_1i_2,i_1i_2} = \frac{1}{2} $ if $ i_1\ne i_2 \in [4]$, and $(\Pi_{2})_{i_1i_1,i_1i_1} = 1 $ for $i_1\in [4]$. 
	We now calculate the matrix $B$: for $i_1,i_2\in [4]$, 
	\begin{align*}
		B_{i_1i_2,i_1i_2} &= B_{i_1i_2,i_2i_1} = \frac{1}{(r!)^2} \sum_{\sigma \in \mathfrak{S}_{r}}\sum_{\tau \in \mathfrak{S}_{r}}  (I_5^{\otimes r})_{(i_{1},i_{2},5,\ldots,5)^{\sigma}, (i_{1},i_{2},5,\ldots,5)^{\tau}}\\
		&= 2(\Pi_{2})_{i_1i_2,i_1i_2} \frac{(r-2)!}{(r!)^2} \sum_{\sigma \in \mathfrak{S}_{r}}  (I_5^{\otimes r})_{(i_{1},i_{2},5,\ldots,5)^{\sigma}, (i_{1},i_{2},5,\ldots,5)^{\sigma}}\\
		&= 2(\Pi_{2})_{i_1i_2,i_1i_2}\frac{(r-2)!}{r!} =  \binom{r}{2}^{-1} (\Pi_2)_{i_1i_2,i_1i_2}. 
	\end{align*}	
	Here, in the second line, we   used the fact that $ I_n^{\otimes r} $ is diagonal and that the row and column indices are equal if and only if $ (i_{\sigma(1)},i_{\sigma(2)}) = (i_{\tau(1)},i_{\tau(2)}) $. 
	Given any permutation $ \sigma \in \mathfrak{S}_r$, there are $ (r-2)! $ permutations $ \tau \in \mathfrak{S}_r $ with this property if $ i_1\ne i_2 $, and $ 2(r-2)! $ permutations $ \tau $ with this property if $ i_1 = i_2 $. Hence, there are $ 2(r-2)!(\Pi_{2})_{i_1i_2,i_1i_2} $ permutations $ \tau $ with this property for any given $ \sigma $. 
	In the third line, we used the fact that all diagonal entries of $ I_n^{\otimes r} $ are equal to $ 1 $, so the sum evaluates to $ r! $.  
	Summarized, we conclude that $ B = \binom{r}{2}^{-1} \Pi_2^{(4)} $. 
	Now, let us calculate $ A $: for $i_1,i_2,j_1,j_2\in [4]$,
	\begin{align*}
		A_{i_1i_2,j_1j_2} = \frac{1}{(r!)^2} \sum_{\sigma \in \mathfrak{S}_{r}}\sum_{\tau \in \mathfrak{S}_{r}}  (Q\otimes I_5^{\otimes r-2})_{(i_{1},i_{2},5,\ldots,5)^{\sigma}, (j_{1},j_{2},5,\ldots,5)^{\tau}}.
	\end{align*}
	Recall that $ Q_{i,5} = 0 $ if $ i\in \{1,\ldots,4\} $. Therefore, the only nonzero contributions come from permutations $ \sigma $ and $ \tau $, which keep all 5-indices `to the right', i.e., at the last $r-2$ positions. There are $ 2(r-2)! $ permutations $ \sigma $ and $ 2(r-2)! $ permutations $ \tau $ with this property. They may be written as $ \sigma = \sigma_2\circ \sigma_{r-2} $ and $ \tau =  \tau_2\circ \tau_{r-2}$, respectively, with $ \sigma_2, \tau_2 \in \mathfrak{S}_2 $ and $ \sigma_{r-2}, \tau_{r-2} \in \mathfrak{S}_{r-2} $.
	So, we obtain  
	\begin{align*}
		A_{i_1i_2,j_1j_2} &= \frac{((r-2)!)^2}{(r!)^2} \sum_{\sigma_2, \tau_2 \in \mathfrak{S}_{2}} (Q\otimes I_n^{\otimes r-2})_{(i_{\sigma_2(1)},i_{\sigma_2(2)},5,\ldots,5), (j_{\tau_2(1)},j_{\tau_2(2)},5,\ldots,5)}.\\
		&= \left(\frac{(r-2)!}{r!}\right)^2 \sum_{\sigma_2, \tau_2 \in \mathfrak{S}_{2}} Q_{i_{\sigma_2(1)}i_{\sigma_2(2)}, j_{\tau_2(1)}j_{\tau_2(2)}}.\\
		&= \binom{r}{2}^{-2} Q_{i_1i_2,j_1j_2}.
	\end{align*}
	In the first line, we used the fact that the $ (r-2)!^2 $ permutations of the 5-indices all yield the same indices. In the second line, we used that $ (I_n^{\otimes r-2})_{5\ldots5,5\ldots5} = 1 $. In the third line, we used that $ Q $ is maximally symmetric, hence in particular doubly symmetric. Thus, the $ 2^2 = 4 $ permutations of the non-5 indices all yield the same entry. Summarizing, we have shown  
	$$ A = \binom{r}{2}^{-2}Q^{(4)},\quad  B = \binom{r}{2}^{-1}\Pi_{2}^{(4)}.
	$$
	 We conclude that the   matrix $ H $ from (\ref{eqmatH})-(\ref{eqmatH2}) can be expressed as
	\begin{align*}
		\binom{r}{2}^{2}  H = Q^{(4)} + \lambda_{\max}(M)|\spec_r|\binom{r}{2}\Pi_{2}^{(4)}. 
	\end{align*}
	Now, we use that $H\succeq 0$ and $\begin{pmatrix} \Pi_2^{(4)} & 0 \\ 0^{T} & 0 \end{pmatrix} \preceq I_5^{\otimes 2} \preceq \frac{1}{\lambda_{\min}(M)} M $ to obtain  
	\begin{align*}
		0\preceq \binom{r}{2}^{2} \begin{pmatrix} H & 0 \\ 0^{T} & 0 \end{pmatrix}  \preceq Q + |\spec_r|\frac{\lambda_{\max}(M)}{\lambda_{\min}(M)} \binom{r}{2} M. 
	\end{align*}
	By definition, $ |\spec_2| $ is the smallest value such that 
	\begin{align*}
		Q + |\spec_2| M \succeq 0. 
	\end{align*}
	Hence, it follows that 
	\begin{align*}
		 |\spec_r|\frac{\lambda_{\max}(M)}{\lambda_{\min}(M)} \binom{r}{2} \ge |\spec_2|, \:\: \text{ and thus }\:\: 	|\spec_r| \ge \frac{\lambda_{\min}(M)}{\lambda_{\max}(M)} |\spec_2|\binom{r}{2}^{-1} = \Omega\left(\frac{1}{r^2}\right),		 
	\end{align*}
\tcolblue{using the fact that $M\succ 0$ \cite{Johnston_Lovitz_Vijayaraghavan_2023} and $|\spec_2|>0$ as $p$ is not sos.} 
\end{proof}

As an application, we also obtain a complexity lower bound for the approximation result in Theorem \ref{thm:banded-qdf}.

\begin{corollary}[Complexity lower bound for the banded real QdF]
	There exists a unit vector $ v\in S^r(\R^n) $   such that, for all scalars $ \alpha \in \R $ and  $ \tau \in \mathcal{C}_d(\R^n) $,   one has
	\begin{align}\label{eq:banded-qdf-lower-bound}
		\|\alpha \tau - \Pi_{\maxsym}(\tr_{r-2} vv^{T})\|_{1} = \Omega(1/r^2).
	\end{align}
	In other words, the convergence rate in \Cref{thm:banded-qdf} is lower bounded by $ \Omega(1/r^2) $.
\end{corollary}

\begin{proof}
	We may assume that $ \frac{1}{2}\alpha_d \le \alpha \le 2$.
	For, if not, then \Cref{rem:bandsize-optimality} shows that  $\|\alpha \tau - \PiMS(\Tr_{r-d}(vv^T))\|_1 $ is lower bounded by a constant and thus, \eqref{eq:banded-qdf-lower-bound} holds trivially. 
	Hence, $ \alpha $ is bounded by constants, which only depend on $ d $. 
	Consider the Choi-Lam form in (\ref{eqpCL}), viewed as a 5-variate polynomial $p$ (so $n=5$, $d=2$), and let $v\in S^2(\R^5)$ be an optimal solution to the program (\ref{eqnu2}) defining the spectral bound $\spec_r(p)$.
	Consider relation (\ref{eq:jlv-estimate-chain}) applied to it, whose  last line reads
	\begin{align*}
	|\spec_r|=p_{\min}-\spec_r & \le \|Q(p)-\spec_r I_5^{\ot d}\|_\infty \cdot
	\|\tau - \alpha^{-1} \PiMS(\Tr _{r-2}(vv^T))\|_1 \\
	& \le \gamma(Q(p)) \cdot  |\alpha^{-1}| \cdot  \|\alpha \tau -   \PiMS(\Tr _{r-2}(vv^T))\|_1,
	\end{align*}
	using (\ref{eqQinfty}) for   the last inequality. By Theorem \ref{thm:lower-bound-spectral}, we know that $|\spec_r|=\Omega(1/r^2)$, which  concludes the proof.
\end{proof}

\subsection{\MoL{The spectral hierarchy does not have generic finite convergence}}
\label{sec:no-finite-convergence}

A nice property of the moment-sos hierarchy is that it has {\em generic} finite convergence, as shown by Huang \cite{Huang2023} (and by Nie \cite{Nie_2014} for generic semialgebraic sets).
That is, the property: $\sos_r(p)=p_{\min}$ for some $r\ge d$, holds for ``almost all" polynomials $p$. More precisely, given $n,d\in \oN$, the set of   polynomials $p\in \oR[x_1,\ldots,x_n]_{2d}$ for which this property holds is Zariski-open.

Recall that the Zariski topology for a finite-dimensional vector space $\oR^N$ has as closed sets the sets $\{y\in\oR^N: F_1(y)=\cdots=F_m(y)=0\}$ for some polynomials $F_1,\ldots,F_m\in \oR[y_1,\ldots,y_N]$ and $m\in \oN$. We will need the well-known fact that each nonempty Zariski-open set is open and dense with respect to both the Euclidean and the Zariski topology.  

On the other hand, we now show that the spectral hierarchy $\spec_r(p)$ does not have  generic finite convergence. For this, consider the set 
\begin{align}\label{eqAnd}
A_{n,d}=\{p\in R_{2d}: \spec_r(p)<p_{\min}\ \text{ for all } r\ge d\}
\end{align}
consisting of the $n$-variate degree $2d$ forms for which the spectral hierarchy {\em does not have}  finite convergence. So, showing generic finite convergence amounts to showing that the set $A_{n,d}$ is Zariski-closed. We will use later the following easy fact: for any $p\in R_{2d} $ and $\lambda\in \oR$, $p\in A_{n,d}$ if and only if $p+\lambda s^d\in A_{n,d}$.

We will show that in the case $d=2$ the set $A_{n,d}$ is not Zariski-closed, by exhibiting a Euclidean ball that is contained in it. We will show the result first for the case $n=2$ (Theorem \ref{thm:no-finite-convergence}) and then extend it to $n\ge 2$ (Theorem \ref{thm:no-finite-convergence-n}). 
For this, we first establish a necessary condition for finite convergence, that depends only on the maximally symmetric matrix $Q(p)$ representing  $p$ (and holds for any $n,d$).

\begin{proposition}\label{prop:finite-conv-implies-eigenvalue}
	Let $ p\in R_{2d}$. Assume that $ \spec_{r}(p) = p_{\min}$ for some $ r\ge d$. Then, for each minimizer $ a\in \mathbb S^{n-1} $ of $ p $, we have $(Q(p)-p_{\min}Q(s^d))a^{\ot d}=0$.
\end{proposition}

\begin{proof}
	Assume $ \spec_{r}(p) = p_{\min}$ for some $r\ge d$. Then, $ Q_r(p) - p_{\min}M_r\succeq 0 $ holds. For each minimizer $a\in \mathbb S^{n-1}$, we have
	\begin{align*}
		(a^{\otimes r})^{T} (Q_r(p) - p_{\min}M_r) a^{\otimes r} = p(a) - p_{\min} = 0.
	\end{align*}  
	Since $ Q_r(p) - p_{\min}M_r $ is positive semidefinite, it follows that $ a^{\otimes r} $ lies in the kernel of $ Q_r(p) - p_{\min}M_r $. From the definition of $ Q_r(p) $, we conclude that 
	\begin{align*}
		0 = \Pi_r ((Q(p) - p_{\min} Q(s^d)) \otimes I_n^{\otimes r-d}) a^{\otimes r} =
		 \Pi_r ((Q(p) - p_{\min} Q(s^d)) a^{\otimes d}  \otimes a^{\otimes r-d}). 
	\end{align*}
	Setting $v=(Q(p) - p_{\min} Q(s^d)) a^{\otimes d}$, we have 
	$\Pi_r(v\ot a^{\ot r-d})=0$. As we now observe, this implies $v=0$. Indeed, for any $x\in \oR^n$, we have 
	$$0=\langle \Pi_r(v\ot a^{\ot r-d}),x^{\ot r}\rangle=\langle v\ot a^{\ot r-d}, x^{\ot r}\rangle=
	 \langle v,x^{\ot d}\rangle \langle a,x\rangle ^{r-d}.$$
	Hence, $\langle v,x^{\ot d}\rangle =0$ for any $x$ that is not orthogonal to $a$ and thus for all $x\in \oR^n$ (by continuity). Since $v$ is a symmetric tensor, this implies in turn that $v=0$, which concludes the proof.	
\end{proof}

This motivates introducing the following sets
\begin{align*}
& S_{n,d}=\{p\in R_{2d}: p_{\min}=0\},\\
& T_{n,d}=\{p\in S_{n,d}: a^{\ot d}\in \ker Q(p)\ \text{ for some } a\in \mathbb S^{n-1} \text{ such that } p(a)=0\}.\end{align*}
By Proposition \ref{prop:finite-conv-implies-eigenvalue}, if the spectral hierarchy has finite convergence for $p\in S_{n,d}$, then $p\in T_{n,d}$. In fact, $p$ belongs  to the smaller set, where one would require $a^{\ot d}\in \ker Q(p)$ for {\em all} zeros $a\in\mathbb S^{n-1}$ of $p$. However, with the chosen definition   of the set $T_{n,d}$,  one can  show that it is a closed set, a crucial property used later to show Theorem \ref{thm:no-finite-convergence}.

\begin{lemma}\label{lem:S-closed}
	The set  $ T_{n,d}$ is closed with respect to  the Euclidean topology on $R_{2d}$. 
\end{lemma}

\begin{proof}
Consider a sequence of polynomials $p_k\in T_{n,d}$ that has a limit $p\in R_{2d}$. As each $p_k$ is nonnegative it follows that $p$ is nonnegative. For each $k$ there exists  $a_k\in \mathbb S^{n-1}$ such that $p_k(a_k)=0$ and $Q(p_k)a_k^{\ot d}=0$. By compactness of the sphere, the sequence $(a_k)_k$ has an accumulation point $a\in \mathbb S^{n-1}$; replacing the sequence by a subsequence we may assume $a$ is the limit of $(a_k)_k$. Then, $0=p_k(a_k)$ converges to $p(a)$, showing $p(a)=0$ and thus $p\in S_{n,d}$. As the matrices $Q(p_k)$ depend linearly on the coefficients of the polynomials $p_k$, it follows that $0=Q(p_k)a_k^{\ot d}$ converges to $Q(p)a^{\ot d}$. Hence, $Q(p)a^{\ot d}=0$, which shows $p\in T_{n,d}$, which concludes the proof. 
\end{proof}

\begin{theorem}\label{thm:no-finite-convergence}
For $n=d=2$, the set $A_{n,d}$ from (\ref{eqAnd}) is not Zariski-closed. That is, the spectral hierarchy does not have generic finite convergence.
\end{theorem}	

\begin{proof}
Consider the bivariate polynomial 
\begin{align}\label{eqf2}
 f = x_1^4 + 2x_2^4 - x_1^2x_2^2 -2x_1x_2^3 = (x_1x_2-x_2^2)^2 + (x_1^2-x_2^2)^2 \in R_4.
\end{align}
So, $f_{\min}=0$ and $f$ has two minimizers $\pm{1\over \sqrt 2} (1,1)$ in the sphere $\mathbb S^1$. Hence, $f\in S_{2,2}$. Moreover, $f\not\in T_{2,2}$. Indeed, its maximally symmetric matrix representation reads
	$$  Q(f) = \begin{pmatrix} 1 &0 & 0 &-1/6\\ 
		0 &-1/6 & -1/6 &-1/2\\
		0 &-1/6 & -1/6 &-1/2\\
		-1/6 &-1/2 & -1/2 &2	
	\end{pmatrix}$$ 
(indexing it   by $ x_1^2, x_1x_2, x_2x_1, x_2^2 $),  and we have $Q(f)a^{\ot 2}\ne 0$ for   $a=(1,1)$.  Hence, $f\in S_{2,2}\setminus T_{2,2}$. Since the set $T_{2,2}$ is closed (by Lemma \ref{lem:S-closed}), there exists $\epsilon>0$ such that the ball $B_f(\epsilon)$ is contained in the complement of $T_{2,2}$, i.e., $B_f(\epsilon)\cap T_{2,2} =\emptyset$. Here, $B_f(\epsilon)=\{p\in R_4: \|p-f\|<\epsilon\}$ denotes the open ball   with center $f$ and radius $\epsilon$, with respect to the $\infty$-norm  on the sphere.
Hence, in view of Proposition \ref{prop:finite-conv-implies-eigenvalue}, we have 
$B_f(\epsilon)\cap S_{2,2}\subseteq A_{2,2}$. 
We now claim that 
\begin{align}\label{eqincl2}
B_f(\epsilon/2)\subseteq B_f(\epsilon)\cap S_{2,2} +\oR s^2 \subseteq A_{2,2},
\end{align}
where $B_f(\epsilon/2)$ is the   ball in $R_4$ centered at $f$ with radius $\epsilon/2$. From this follows immediately that the set $A_{2,2}$ is not Zariski-closed, which concludes the proof of the theorem.

The right most inclusion in (\ref{eqincl2}) follows from the fact (observed earlier) that finite convergence of the spectral hierarchy is preserved under shifting by a scalar multiple of the constant polynomial $s^2$. We now verify the left most inclusion. For this, let $p\in B_f(\epsilon/2)$, i.e., $\|p-f\|<\epsilon/2$. By evaluating at a minimizer $a\in\mathbb S^1$ of $f$, we get $|p(a)|<\epsilon/2$, implying $p_{\min}\le p(a)<\epsilon/2$. By evaluating at a minimizer $b\in \mathbb S^1$ of $p$, we get 
$-p_{\min} \le f(b)- p(b)<\epsilon/2$, and thus $|p_{\min}|<\epsilon/2$. Using the triangle inequality, we have
$\|p-p_{\min}s^2-f\|\le \|p-f\|+ |p_{\min}| < \epsilon/2+\epsilon/2=\epsilon$. This shows that 
$p-p_{\min}s^2\in B_f(\epsilon)$,  and thus $p\in B_f(\epsilon)\cap S_{2,2}+\oR s^2$, as desired.
\end{proof}

\begin{remark}\label{rem:no-finite-convergence}
Observe that the proof of Theorem \ref{thm:no-finite-convergence} relies   on the existence of a polynomial $f\in S_{2,2}\setminus T_{2,2}$. In other words, in order to show that the spectral hierarchy does not have generic finite convergence   for $n$-variate forms of degree $2d$, it suffices to exhibit one concrete form $F\in S_{n,d}\setminus T_{n,d}$. As we next observe, one can easily construct such a form of degree $4$ for any $n\ge 2$, simply by viewing the bivariate form $f$ from (\ref{eqf2}) as an $n$-variate form. 
{We expect that one can construct such $F\in S_{n,d}\setminus T_{n,d}$ and thus show non-genericity of finite convergence of the spectral bounds for any $d\ge 2$,  we leave showing this to further research.}
\end{remark}


\begin{theorem}\label{thm:no-finite-convergence-n}
For $n\ge 2$ and $d=2$, the set $A_{n,d}$ from (\ref{eqAnd}) is not Zariski-closed. That is, the spectral hierarchy does not have generic finite convergence.
\end{theorem}

\begin{proof}
Let $F$ denote  the $n$-variate degree 4 form   obtained by viewing the bivariate form $f$ from (\ref{eqf2}) as an $n$-variate form. Then, its maximally matrix representation has a block-form
$$\MaxSym(F)=\left(\begin{matrix} 
\MaxSym(f) & 0 \cr 0 & 0 \end{matrix}\right),$$
where $\MaxSym(f)$ is indexed by the pairs $(i,j)\in [2]^2$ and the zero columns and rows are indexed by the pairs $(i,j)\in [n]^2$ with $i\ge 3$ or $j\ge 3$ (thus the analog situation as for the matrices in (\ref{eqmatMQ})). Clearly, $F_{\min}=f_{\min}=0$ and thus  $F\in S_{n,2}$. Moreover,  $F\not\in T_{n,2}$. For this, let $b\in\mathbb S^{n-1}$ be a zero of $F$; we show that $b^{\ot 2}\not\in \ker \MaxSym(F)$. For the vector $a=(b_1,b_2)^T$, we have  $f(a)=F(b)=0$ and thus $a$ is a scalar multiple of $(1,1)^T$. Hence, $\MaxSym(f)a^{\ot 2}\ne 0$, which implies $\MaxSym(F)b^{\ot 2} \ne 0$. 
 
As observed in Remark \ref{rem:no-finite-convergence}, the proof of Theorem \ref{thm:no-finite-convergence} can now be followed mutadis mutandis to show that the set $A_{n,2}$ is not Zariski-closed, showing that the spectral bounds do not have generic finite convergence for $n$-variate degree 4 forms. 
\end{proof}

\section{Further discussion}\label{secfinal}

\subsection*{Connections to the moment problem and polynomial optimization} 

{Real} de Finetti representations are elements of the set $ \mathcal{C}_d(\R^n) $ from (\ref{eq:Cd-definition}), i.e., convex combinations of matrices $ (u_iu_i^{T})^{\otimes d} $, where all $ u_i $ are real unit vectors.  
Elements of $ \mathcal{C}_d(\R^n) $ are maximally symmetric matrices, whereas elements of $ \mathcal{C}_d(\C^n) $ are in general not maximally symmetric. 
Therefore, the quantum de Finetti theorem cannot be transferred verbatim to real doubly symmetric matrices, as we saw in \Cref{prop:nonexistence-of-real-qdf}. Nevertheless, we saw that it is possible to make statements about maximally symmetric matrices (\Cref{theoDWr2}) and about real doubly symmetric matrices (\Cref{thm:banded-qdf}). The relevance of such results to polynomial optimization on the sphere should be clear: Doherty and Wehner \cite{Doherty_Wehner_2013} use their real QdF (Theorem \ref{theoDW}) to analyze the sos bounds; Lovitz and Johnston \cite{Johnston_Lovitz_Vijayaraghavan_2023} use the QdF (Theorem \ref{theoQdF}) to analyze their spectral bounds;  we saw how to analyze these spectral bounds  using \MoL{our real banded QdF}
(\Cref{thm:banded-qdf}).

The \emph{moment problem} addresses a fundamental question in real algebraic geometry and functional analysis  (see \cite{Schmudgen1991}), namely,    when a linear functional on the polynomial space has an (exact or approximate) representing measure, possibly with a given support. Real de Finetti representations address the approximate moment problem for the space of homogeneous polynomials and measures supported \emph{on the sphere}. So, the result of \Cref{theoDWr2} can be roughly paraphrased as follows (by combining  with \Cref{propsosP}):
If $L$ is a linear functional acting on degree $2r$ forms that is nonnegative on sums of squares \MoL{and satisfies  $L(s^r)=1$}, then for   $d\le r$ there exists a measure $\mu$ on the sphere whose (low degree) moments approximate well  the `pseudo-moments'  of $L$ up to degree $2d$, with an error term in $  {O}(1/r^2) $. 
In other words, real de Finetti theorems permit to show \MoL{quantitative results for the approximate moment problem on the sphere.}
This question has been addressed recently for the moment problem on general compact semi-algebraic sets by Baldi and Mourrain \cite{Baldi_Mourrain_2023} \MoL{(see also \cite{Baldi-Mourrain-Parusinski})}, who show quantitative results with polynomial dependency  of the form $O(1/r^c)$, for some unknown constant $c>0$. Showing an explicit polynomial dependency of the form $O(1/r^2)$ for some highly symmetric sets like the ball or the simplex is the subject of further research.

\begin{table}[H]
	\begin{tabular}{|c|c|c|c|}
		Result & Scope & complexity ub  & complexity lb \\
		\hline
		Diaconis-Freedman \cite{Diaconis_Freedman_1980_deFinetti}& probability distr.  & $ {O}(1/r)$ & $\Omega(1/r)$ \\
		Christandl et. al. \cite{Christandl_2007_deFinetti_oneandhalf} & doubly sym.   ($\C$)& $ {O}(1/r)$ & $\Omega(1/r)$ \\
		Doherty-Wehner \cite{Doherty_Wehner_2013}& max. sym.     ($\R$)  & $ {O}(1/r)$ & \textbf{none} \\
		\Cref{theoDWr2} & max. sym.    ($\R$) & $ {O}(1/r^2)$ &  \textbf{none} \\
		\Cref{thm:banded-qdf} & doubly sym.    ($\R$) & $ {O}(1/r)$ &  $ \Omega(1/r^2) $ \\
	\end{tabular}
	\smallskip \caption{Convergence rate for  de Finetti approximation results:  in rows 2-4, the second column concerns a matrix $M$, which is positive semidefinite with trace 1,   and has additional symmetry and  entries in $\C$ or $\R$ as indicated. The last two columns give upper and lower bounds for the complexity of the convergence rate.}\label{tab:de-finetti-comparison}
\end{table}

\subsection*{Complexity questions}
In Table \ref{tab:de-finetti-comparison} we gather known results on the convergence rate of the various de Finetti type results. As indicated there, no complexity lower bound is known for the result of \Cref{theoDWr2} and it would also be interesting to close the gap between $O(1/r)$ and $\Omega(1/r^2)$ for \Cref{thm:banded-qdf}.
Accordingly, it is also  an open question to settle what is the exact  convergence rate for the sos lower bounds $\sos_r(p)$ and the spectral bounds $\spec_r(p)$. For the spectral bounds, a complexity upper bound in $O(1/r )$ is known  \cite{Johnston_Lovitz_Vijayaraghavan_2023} and a complexity lower bound in $\Omega(1/r^2)$ is shown in this paper. 
For the sos bounds, only an upper bound in $O(1/r^2)$ is known from \cite{Fang_Fawzi_2020}, but {\em no complexity lower bound} is known as of today. \MoL{A notable result is for polynomial optimization on the hypercube $[-1,1]^n$, where Baldi and Slot \cite{Baldi-Slot} show a convergence analysis for the sos bounds in $O(1/r)$  and  a complexity lower bound in $\Omega(1/r^8)$.}

\subsection*{Acknowledgments}

The authors were supported by the Dutch Scientific Council (NWO)
grant OCENW.GROOT.2019.015 (OPTIMAL). A. Blomenhofer was also supported by the European Research Council (ERC) under Agreement 818761. 

A slightly shortened version of this manuscript (without Appendices A,B and D) has been accepted for publication in {\em SIAM Journal on Optimization.} We are grateful to the referees for their careful reading and many helpful suggestions that helped us improve the presentation as well as some of the results. In particular, we thank a referee for raising the question about generic finite convergence, which inspires Section \ref{sec:no-finite-convergence}. We also thank the journal editor for providing constructive feedback. Furthermore, we thank Benjamin Lovitz, Markus Schweighofer, and Lucas Slot for useful discussions.

\bibliography{bibPoP}
\bibliographystyle{acm}

\appendix
\section{\tcolblue{Some examples}}\label{appendix-example}

We give here a few small examples to illustrate the tensor notions that have been introduced in Section \ref{sec:prelims}.
First, we give an example to illustrate how permutations work on tensors: consider $v\in(\R^2)^{\ot 2}$ and the transposition $\sigma=(12)$, then, with respect to the ordering of $[2]^2$ as $11,12,21,22$
(corresponding to the standard basis $e_1\ot e_1,e_1\ot e_2,e_2\ot e_1,e_2\ot e_2$ of $(\R^2)^{\ot 2}$), we have
\begin{align*}
v=\left(\begin{matrix}v_{11} \cr v_{12} \cr v_{21} \cr v_{22}\end{matrix}\right),\quad v^{(12)}= \left(\begin{matrix}v_{11} \cr v_{21} \cr v_{12} \cr v_{22}\end{matrix}\right),\quad
\Pi_2(v)={1\over 2}(v+ v^{(12)}) =\left(\begin{matrix}v_{11} \cr (v_{12}+v_{21})/2 \cr (v_{12}+v_{21})/2 \cr v_{22}\end{matrix}\right).
\end{align*}
Indeed, we have $v^{(12)}_{i_1i_2}=v_{i_2i_1}$ for any  $i_1,i_2\in [2]$, which gives $v^{(12)}_{11}=v_{11},$ $ v^{(12)}_{22}=v_{22}$,  $v^{(12)}_{12}=v_{21},$ and $ v^{(12)}_{21}=v_{12}.$

As another example, for the tensor $v=e_1\ot e_1\ot e_2\ot e_2\in (\R^2)^{\ot 4}$, its projection $\Pi_4(v)$  onto the symmetric subspace $S^4(\R^2)$ reads
\begin{align}\label{eqexPi4}
\begin{split}
&\Pi_4(e_1\ot e_1\ot e_2\ot e_2)=
{1\over 6} \big(e_1\ot e_1\ot e_2\ot e_2 + e_1\ot e_2\ot e_1\ot e_2+\\
&e_1\ot e_2\ot e_2\ot e_1 + e_2\ot e_1\ot e_1\ot e_2 + e_2\ot e_1\ot e_2\ot e_1+ e_2\ot e_2\ot e_1\ot e_1\big),
\end{split}\end{align}
which can be checked by enumerating the permutations of $[4]$ and their action on $v$.

\medskip
Consider the matrix $I_2^{\ot 2} \in \R^{[2]^2\times [2]^2}$,  acting on $(\R^2)^{\ot 2}$. {It is not doubly symmetric (since, e.g., its $(12,12)$-entry equals 1 while its $(21,12)$-entry equals 0). Hence, it is also not maximally symmetric.}
{For illustration, we show below its projection $\Pi_2I_2^{\ot 2}\Pi_2=\Pi_2$ onto the space of doubly symmetric matrices  and its projection $\PiMS(I_2^{\ot 2})$ onto the space of maximally symmetric matrices (as given in relation (\ref{ex:IdMS})):
\begin{align}\label{eqprojI2} 
\Pi_2=\left( \begin{matrix} 1 & 0 & 0 & 0\cr
0 & 1/2 & 1/2 & 0\cr
0 & 1/2 & 1/2 & 0\cr
0 &0 &0 & 1
\end{matrix}\right), \quad
\PiMS(I_2^{\ot 2}) =\left(\begin{matrix} 1 & 0 & 0 & 1/3\cr
0 & 1/3 & 1/3 & 0\cr
0 & 1/3 & 1/3 & 0\cr
1/3 & 0 & 0 & 1\end{matrix}\right),
\end{align}
where   entries are indexed by $11,12,21,22$. }
To obtain the projection $\PiMS(I_2^{\ot 2})$ onto the space of maximally symmetric matrices, 
one can proceed as follows:
\begin{align*}
I_2^{\ot 2}= (e_1e_1^T+e_2e_2^T)\ot (e_1e_1^T+e_2e_2^T)& 
= (e_1\ot e_1) (e_1\ot e_1)^T + (e_2\ot e_2) (e_2\ot e_2)^T \\
& + (e_1\ot e_2) (e_1\ot e_2)^T + (e_2\ot e_1) (e_2\ot e_1)^T.
\end{align*}
Then, the matrix $(e_i\ot e_i)(e_i\ot e_i)^T$ is maximally symmetric for $i=1,2$. To obtain $\PiMS( (e_1\ot e_2) (e_1\ot e_2)^T)$, consider the associated tensor $e_1\ot e_2\ot e_1\ot e_2\in (\R^2)^{\ot 4}$, whose projection onto the symmetric subspace $S^4(\R^2)$ satisfies $\Pi_4(e_1\ot e_2\ot e_1\ot e_2)= \Pi_4(e_1\ot e_1\ot e_2\ot e_2)$, whose explicit value is given in (\ref{eqexPi4}). By reshaping this tensor as a matrix we get
$$
\PiMS((e_1\ot e_2) (e_1\ot e_2)^T)
={1\over 6} \left(\begin{matrix} 0 & 0 & 0 & 1\cr 0 & 1 & 1 & 0\cr 0& 1 & 1 & 0\cr 1& 0 & 0 & 0\end{matrix}\right).
$$
Clearly,
$\PiMS((e_2\ot e_1) (e_2\ot e_1)^T)=  \PiMS((e_1\ot e_2) (e_1\ot e_2)^T)$, and from this one gets the desired expression  for $\PiMS(I_2^{\ot 2})$ in (\ref{eqprojI2}) (and (\ref{ex:IdMS})).

To illustrate  the partial trace operation,  consider a matrix 
$$M=\left(\begin{matrix} A_{11} &\ldots &A_{1n}\\ \vdots &\cdots &\vdots \\  A_{n1} &\ldots &A_{nn}  \end{matrix}\right)\in \K^{[n]^2\times [n]^2}, \quad \text{ where } A_{ij} \in
\K^{n\times n}$$
and the entries of $M$ are indexed by $11,12,\ldots,1n,\ldots, n1,n2, \ldots, nn$. Then, we have
$$\tr_1(M)=\left(\begin{matrix} \tr(A_{11}) &\ldots &\tr(A_{1n})\\ \vdots &\cdots &\vdots \\  \tr(A_{n1}) &\ldots &\tr(A_{nn})  \end{matrix}\right)\in \K^{ n\times n},$$
where  we `trace out the second subsystem' to get $\tr_1(M)$. Recall that the index $ k $ in $ \Tr_k $ denotes the number of subsystems from the right, which are traced out. Hence, $ \Tr_2(M) = \Tr(M) $ is a scalar.

\section{\tcolblue{Proof of Lemma \ref{lemJL}}} \label{appendix-proof-lemma}

The result from Lemma \ref{lemJL} is essentially extracted from the proof of \cite[Theorem 4.1]{Johnston_Lovitz_Vijayaraghavan_2023}. Since it plays a crucial role to show \Cref{thm:banded-qdf}, we give a sketch of proof for clarity of exposition.
Assume $\tau\in \mathcal C_d(\C^n)$ with $\Tr(\tau )=1$, and set $\alpha:= \Tr(\PiMS(\tau))$ and $\alpha_d= {2^d\over {2d\choose d}}$.	We show that ${1\over \alpha}\PiMS(\tau) \in \mathcal C_d(\R^n)$  and $\alpha_d\le \alpha\le 1.$

Clearly, it suffices to show the result for $\tau=(uu^*)^{\ot d}$, where $u\in \C^n$ is a unit vector. Write $u=a+\bfi b$, where $a,b\in\R^n$ with $\|a\|^2+\| b\|^2=1$. By Lemma~\ref{lem:ppt-sym-is-real-rank-2}, $\PiMS(\tau)=\PiMS((aa^T+bb^T)^{\ot d})$ and thus $\alpha=\Tr(\PiMS((aa^T+bb^T)^{\ot d}))$.
The next step is  to exploit a result of Reznick \cite[Corollary 5.6]{Reznick_2013} for showing that ${1\over \alpha} \PiMS( (aa^T+bb^T)^{\ot d}))\in \mathcal C_d(\R)$.
	
Reznick \cite{Reznick_2013} shows that the bivariate homogeneous polynomial $ (t_1^2 + t_2^2)^d $ in variables $ t_1, t_2 $ can be written as a sum of powers of linear forms. More precisely, he gives explicit  real unit vectors $ v_k\in \R^2 $ ($k=0,1,\ldots,d$) such that 
\begin{align*}
	(t_1^2 + t_2^2)^d = \frac{\alpha_d\cdot 2^d}{d+1} \sum_{k = 0}^{d} \big\langle v_k, \begin{pmatrix} t_1 \\ t_2 \end{pmatrix} \big\rangle^{2d}.
\end{align*}	
By plugging in $ t_1 = a^{T}x $ and $ t_2 = b^{T}y $, one obtains that the degree $2d$ $n$-variate form $ ((a^{T}x)^2 + (b^{T}x)^2)^d =\langle (aa^T+bb^T)^{\ot d},(xx^T)^{\ot d}\rangle $ is a sum of $ 2d $-th powers of linear forms in $ x $,   of the form $\sum_{k=0}^d (c_k^Tx)^{2d}=\sum_{k=0}^d \langle (c_kc_k^T)^{\ot d}, (xx^T)^{\ot d}\rangle$ for some $c_k\in\R^n$ (obtained from $v_k,a, b$). So,
	$\langle  (a^{T}x)^2 + (b^{T}x)^2)^d), (xx^T)^{\ot d}\rangle 
	= \langle 	(\sum_{k=0}^d (c_kc_k^T)^{\ot d}, (xx^T)^{\ot d}\rangle. $
Hence, the maximally symmetric representation of the form  $ ((a^{T}x)^2 + (b^{T}x)^2)^d$ is equal to $\sum_{k=0}^d (c_kc_k^T)^{\ot d}$. Therefore, $\PiMS((aa^T+bb^T)^{\ot d})= \sum_{k=0}^d (c_kc_k^T)^{\ot d}$, which  thus belongs to 	the cone generated by $ \mathcal{C}_d(\R^n) $.	
So, after normalizing by the trace $ \alpha $, one obtains that $\PiMS(\tau)/\alpha \in \mathcal{C}_d(\R^n) $.  
 
To show the lower bound $ \alpha \ge \alpha_d$, the authors of \cite{Johnston_Lovitz_Vijayaraghavan_2023} use the eigenvalue decomposition to write $ aa^{T} + bb^{T} = t vv^{T} + (1-t) ww^{T} $ with $t\in [0,1]$ and  {orthogonal unit vectors } $ v $ and $ w $. 
They consider the function $ \varphi_{d}(t) = \Tr(\PiMS((tvv^T+(1-t)ww^T)^{\ot d}))$, a polynomial in $ t $, so that $\alpha=\varphi_d(t_0)$ for some $t_0\in [0,1]$. They show that it can be written as a sum $ \varphi_{d}(t) = \sum_{k = 0}^{\lfloor d/2\rfloor} a_k (t-1/2)^{2k} $ of even powers of $ (t-1/2) $ with nonnegative coefficients $ a_k \ge 0 $.	
From this, it readily follows that the minimum of $\varphi_d(t)$ for $t\in [0,1]$ is attained at $t=1/2$, with $\varphi_d(1/2) = \tr(\PiMS(I_2^{\ot d}))=2^{2d}{2d\choose d}^{-1}= \alpha_d\cdot 2^d$ (which also follows from \cite{Reznick_2013}). In the same way, the maximum value of $\varphi_d(t)$  is attained at $t\in\{0,1\}$ and is equal to $\varphi_d(1)=1$. 
The desired bounds $\alpha_d\le \alpha\le 1$ follow directly.

\section{\tcolblue{Proof of Proposition \ref{prop:jlv-relaxes-pmin}}}\label{appendix-proof-proposition}

Let $p\in R_{2d}$. Recall the definitions in (\ref{eqQrpM}) for the matrices $Q_r(p)$ and $M_r$ that represent, respectively, the polynomials $ps^{r-d}$ and $s^r$.

First, we show that $\spec_r(p)\le \sos_r(p)$. The matrix $Q_r(p)-\spec_r(p) M_r$ represents the polynomial $s^{r-d}p-\spec_r(p) s^r$ and  it is positive semidefinite (by definition of   $\spec_r(p)$). Hence, the polynomial $s^{r-d}p-\spec_r(p) s^r$ is sos, and thus $\spec_r(p)\le \sos_r(p)$, by \eqref{eqsosr-alternative}. 

We now show that $\spec_r(p)\le \spec_{r+1}(p)$. For this, we use the formulation (\ref{eqnu3}).
 Let $Y$ be an optimal solution for the formulation (\ref{eqnu3}) of the parameter $\spec_{r+1}(p)$; so $Y$ is $(r+1)$-doubly symmetric, $Y\succeq 0$ and $\langle M_{r+1},Y \rangle=1$. Set $X= \Tr_1(Y)$. Then, $X$ is $r$-doubly symmetric and positive semidefinite.
We claim that $\langle M_r, X\rangle =1$. Indeed,  
\begin{align*}
1& =\langle M_{r+1},Y\rangle= \langle \Pi_{r+1} (M\ot I_n^{\ot (r+1-d)})\Pi_{r+1}, Y\rangle 
= \langle M\ot I_n^{\ot (r+1-d)}, \Pi_{r+1}Y\Pi_{r+1}\rangle \\
& = \langle M\ot I_n^{\ot (r+1-d)}, Y\rangle 
= \langle M\ot I_n^{\ot (r-d)}, \Tr_1(Y)\rangle 
=\langle M_r, X\rangle,
\end{align*}
where, in the second line,  we use the fact that $Y$ is doubly symmetric for the first equality, the definition of the partial trace for the second equality, and the fact that $X=\Tr_1(Y)$ is doubly symmetric for the last equality. In the same way, we have $$\langle Q_{r+1}(p),Y\rangle = \langle Q(p)\ot I_n^{\ot (r+1-d)}, Y\rangle =
\langle Q(p)\ot I_n^{\ot (r-d)}, \Tr_1(Y)\rangle = \langle Q_r(p), X\rangle.$$
Hence, $X$ is a feasible solution for the formulation (\ref{eqnu3}) of $\spec_r(p)$ with objective value $\spec_{r+1}(p)$, which shows    $\spec_r(p)\le \spec_{r+1}(p)$.

\section{\tcolblue{The constant   $\gamma(Q(p))$ is not well-behaved}}\label{appendix-bad-gamma}

First, we discuss the relevance of establishing a performance analysis in terms of the range of values $p_{\max}-p_{\min}$ as is done in Theorem \ref{thm:pmax-pmin-spectral}.
This is indeed the common practice in continuous optimization, as discussed, e.g., in \cite{Vavasis1992,deKlerkLaurentParrilo2006}:
given $\epsilon>0$, a value $\val_\epsilon$ is called an $\epsilon$-approximation for the problem of  minimizing a function $p$ on a compact set $S$ if it satisfies $|p_{\min}-\val_\epsilon|\le \epsilon(p_{\max}-p_{\min})$, where $p_{\min}$ and $p_{\max}$ are the minimum and maximum values taken by $p$ on $S$. 
Such definition is motivated by the fact that it enjoys nice invariance properties, like translating the function by an (appropriate) constant or  scaling by a positive constant. 
When $S$ is the sphere, it would thus be desirable to have an error analysis that depends on the range of values $p_{\max}-p_{\min}$.
This is precisely what we do in Theorem  \ref{thm:pmax-pmin-spectral}. Note   that the error analysis for the sos bounds in Theorem \ref{theoFF} does indeed involve the range  $p_{\max}-p_{\min}$.

With this in mind, we investigate the parameter 
$$\gamma(Q(p))=\max\{\lambda_{\max}(Q(p))-\spec_d(p), p_{\min}-\lambda_{\min}(Q(p))\},$$ 
which enters the error analysis  in Theorem \ref{thm:analysis-spectral-hierarchy}. 
We show that this parameter does not behave well under translating $p$ by a multiple of $s^d$:
there does not exist a constant $C_{n,d}$ (depending only on $n$ and $d$) such that 
\begin{align}\label{eqconstant}
\gamma(Q(p))\le C_{n,d}(p_{\max}-p_{\min}) \quad \text{ for all } p\in R_{2d}.
\end{align}
Hence, one cannot directly derive  Theorem \ref{thm:pmax-pmin-spectral} from Theorem \ref{thm:analysis-spectral-hierarchy}.

\medskip
For this, given $p\in R_{2d}$ and   $a\in\R$,   consider the polynomial $p(a):= p+as^d\in R_{2d}$. Then, we have
$p(a)_{\max}-p(a)_{\min}=p_{\max}-p_{\min}$,  
$\spec_d(p(a))= \spec_d(p)-a$, and $Q(p(a))= Q(p)+a M$, where $M=\PiMS(I_n^{\ot d})$.
We disprove the existence of a constant $C_{n,d}$ satisfying (\ref{eqconstant}) by showing a lower bound on 
the parameter $\gamma(Q(p(a)))$ that is linear and monotone increasing in $a$  (for all large enough $a$).

We first consider the term $\lambda_{\max}(Q(p(a)))-\spec_d(p(a))$. 
Let $v$ be a unit eigenvector of $M$ for $\lambda_{\max}(M)$. Then,
$\lambda_{\max}(Q(p(a))) \ge v^T Q(p(a))) v= v^T Q(p) v +a \lambda_{\max}(M)$. 
Therefore,
\begin{align}\label{eqRmax}
\begin{split}
 \lambda_{\max}(Q(p(a)))-\spec_d(p(a))& = \lambda_{\max}(Q(p(a))) -\spec_d(p)-a\\
&\ge   a (\lambda_{\max}(M) -1) + v^T Q(p) v -\spec_d(p).
\end{split}\end{align}
We now consider the term $p(a)_{\min}-\lambda_{\min}(Q(p(a)))$. By taking a unit eigenvector $w$ of $M$ for $\lambda_{\min}(M)$,   
$\lambda_{\min}(Q(p(a))) \le w^T Q(p(a))w= w^TQ(p)w +a\lambda_{\min}(M)$ holds, and  
\begin{align}\label{eqRmin}
\begin{split}
p(a)_{\min}-\lambda_{\min}(Q(p(a))) &=p_{\min}-a -\lambda_{\min}(Q(p(a)))\\
& \ge a(1-\lambda_{\min}(M)) +p_{\min}-w^T Q(p)w.
\end{split}\end{align}
Therefore, $\gamma(Q(p(a)))\ge \varphi(a)$, where 
$$\varphi(a)= \max\{ a (\lambda_{\max}(M) -1) + v^T Q(p) v -\spec_d(p), a(1-\lambda_{\min}(M)) +p_{\min}-w^T Q(p)w\}$$
is the maximum of the lower bounds in (\ref{eqRmax})-(\ref{eqRmin}). 
{Since  $\lambda_{\min}(M)\le 1\le \lambda_{\max}(M)$ with $\lambda_{\max}(M)-\lambda_{\min}(M)>0$,  we obtain that $\varphi(a)$ is linear monotone increasing in $a$ for all large enough $a$, which concludes the proof.
}

Note that one can analogously show that the parameter $\|Q(p(a))\|_{\infty}$, which appears in  Theorem \ref{theoJLV}, {also has this undesirable feature} that it is monotone increasing for large enough $a$; the proof is the same (even simpler).

\end{document}